%% file: infinite_continuous.tex
\documentclass[11pt, letterpaper]{amsart}
\usepackage{amssymb, hyperref}
\newtheorem{theorem}{Theorem}[section]
\newtheorem{lemma}[theorem]{Lemma}
\newtheorem{proposition}[theorem]{Proposition}
\theoremstyle{definition}
\newtheorem{definition}[theorem]{Definition}

\theoremstyle{remark}
\newtheorem{remark}[theorem]{Remark}

\numberwithin{equation}{section}

\newcommand{\basis}{(\Omega, \mathcal{F} (\cdot),  \bP)}

\newcommand{\Mloc}{\mathsf{M_{loc}}}
\newcommand{\cMloc}{\mathsf{cM_{loc}}}
\newcommand{\cMloci}{\mathsf{cM}^I_{\mathsf{loc}} }

\newcommand{\expec}{\bE}
\newcommand{\Sem}{\mathsf{S}}
\newcommand{\class}{\mathsf{D}}

\include{defs}

\begin{document}

\title[Infinite-dimensional stochastic integration and Mathematical Finance]{Stochastic integration with respect to arbitrary collections of continuous semimartingales and applications to Mathematical Finance}%
\author{Constantinos Kardaras}%
\address{Constantinos Kardaras, Department of Statistics, London School of Economics and Political Science,   10 Houghton Street, London, WC2A 2AE, UK}%
\email{k.kardaras@lse.ac.uk}%

\thanks{The author would like to thank Ioannis Karatzas for valuable comments and help regarding exposition.}%

\date{\today}%
\begin{abstract}
Stochastic integrals are defined with respect to a collection $P = (P_i; \, i \in I)$ of continuous semimartingales, imposing no assumptions on the index set $I$ and the subspace of $\bR^I$ where $P$ takes values. The integrals are constructed though finite-dimensional approximation, identifying the appropriate local geometry that allows extension to infinite dimensions. For local martingale integrators, the resulting space $\Sem(P)$ of stochastic integrals has an operational characterisation via a corresponding set of integrands $\RKH(C)$, constructed with only reference the covariation structure $C$ of $P$. This bijection between $\RKH(C)$ and the (closed in the semimartingale topology) set $\Sem(P)$ extends to families of continuous semimartingale integrators for which the drift process of $P$ belongs to $\RKH(C)$. In the context of infinite-asset models in Mathematical Finance, the latter structural condition is equivalent to a certain natural form of market viability. The enriched class of wealth processes via extended stochastic integrals leads to exact analogues of optional decomposition and hedging duality as the finite-asset case. A corresponding characterisation of market completeness in this setting is provided. 
\end{abstract}

\maketitle


\section*{Introduction}

\subsection*{Discussion}

One of the reasons why the theory of stochastic integration with respect to a finite number of semimartingale integrators $P \equiv (P_i; \, i \in I)$ is comprehensive is that, up to Hilbert isomorphisms, finite-dimensional Euclidean spaces have a unique interesting geometric and topological structure: \cite[Theorem 5.21]{MR2378491}. Predictable integrands $h$ take values in the space of linear functionals of $\bR^I$, and infinitesimal increments $h \ud P$ of stochastic integrals are formally understood as actions of $h$ on $\ud P$. The choice of an inner product (and a basis) on $\bR^I$ only affects the representation (and interpretation) of integrands. Necessary and sufficient conditions---even with predictable characterisation, as in \cite{MR2126973}---exist to ensure that the stochastic integral of a predictable process with respect to $P$ is well defined, and the resulting vector space of all possible stochastic integrals with respect to $P$ is closed in a natural strong semimartingale topology, considered in \cite{MR544800}, and which we shall refer to as $\Sem$-topology. This closedness property is conceptually important, validating in essence that the program of defining stochastic integration has been carried out in a satisfactory way; it is also important in a practical sense: apart from its obvious value in Stochastic Analysis (for example, in the study of stable subspaces of local martingales), it has found applications in other areas of Applied Probability. One such area in Mathematical Finance, where the previous become crucial in cornerstone results of the theory; we shall further elaborate on this later on.

Given arbitrary collections of semimartingales $P \equiv (P_i; \, i \in I)$, restricting attention to the class of stochastic integrals using only a finite number of these integrators typically leads to failure of $\Sem$-closedness. This can be remedied, of course, by considering the closure in $\Sem$-topology of the aforementioned class; thus, one may define abstractly the set $\Sem(P)$ of ``extended stochastic integrals'' with respect to $P$. This approach results both in $\Sem(P)$ being $\Sem$-closed, and avoids complications when dealing with infinite-dimensional state spaces as the ones mentioned in the next paragraph. However, it comes with a considerable price: the abstractly-defined class $\Sem(P)$ has no operational, or structural, characterisation.

A workable construction of stochastic integral in infinite dimensional state spaces involves certain decisions. The vector space $\bR^I$ is deemed too large, and its product topology too weak, for interesting linear pairings of integrands with integrators to exist. Typically, one restricts $P$ to take values in a chosen separable Banach space $\bY$, and $h \ud P$ is again formally interpreted as the local action of a predictable process $h$, with values in linear functionals on $\bY$, on the semimartingale increment $\ud P$. A further decision concerns the subclass of linear functionals that integrands are allowed to take values in. Restricting attention to the class $\bY^*$ of continuous linear functionals may not result in $\Sem$-closedness, and some extension is necessary. For instance, when $P$ is a $\bY$-valued Wiener process  for some Hilbert space $\bY$ (see, for example, \cite[Section 4.1]{MR3236753} for definitions and properties), one has to consider integrands that take values in non-continuous (unbounded) linear functionals defined on a strict subspace $\bX$ of $\bY$; see \cite[Chapters 3--4]{MR2235463}, \cite[Chapter 4]{MR3236753}, as well as \cite[Chapter 5]{MR688144} and \cite{MR1655149}. In infinite-dimensional settings, it is often the case that almost no path of the process $P$ lies on $\bX$, already obscuring the interpretation of $h \ud P$ as $h$ acting on $\ud P$.

For an illustration of the above, let $I$ be countably infinite, and let $P \equiv (P_i; \, i \in I)$ be a collection of independent standard Brownian motions. With weights $(b_i; \, i \in I)$ such that $b_i > 0$, $i \in I$, and $\sum_{i \in I} b_i < \infty$, the fact that $\sum_{i \in I} b_i |P_i|^2$ is a finitely-valued process implies that $P$ takes values in the Hilbert space $\bY = \{ y \in \bR^I \such \sum_{i \in I} b_i |y_i|^2  < \infty \}$ equipped with inner product $\bY \times \bY \ni (y, z) \mapsto \inner{y}{z}_{\bY} \dfn \sum_{i \in I} b_i y_i z_i$. In order for $\int_0^\cdot \inner{\eta}{\ud P}_{\bY} \equiv \int_0^\cdot \sum_{i \in I} b_i \eta_i \ud P_i$ to make sense, it is sufficient that $\eta$ takes values in $\bZ \dfn \{ z \in \bR^I \such \sum_{i \in I} |b_i z_i|^2  < \infty \}$, a strict superset of $\bY^* \simeq \bY$. The Cauchy-Schwarz inequality $\sum_{i \in I} |b_i z_i y_i |\leq \sqrt{\sum_{i \in I} |b_i z_i|^2} \sqrt{\sum_{i \in I} |y_i|^2}$ implies that linear functionals with representation from $\bZ$ do not act on the whole space $\bY$, but rather on the subspace $\bX \dfn \{ y \in \bR^I \such \sum_{i \in I} |y_i|^2  < \infty \}$. In fact, the weights $(b_i; \, i \in I)$ are completely irrelevant: one may simply endow $\bX$ with a Hilbert structure via the inner product $\bX \times \bX \ni (y, z) \mapsto \inner{y}{z}_{\bX} = \sum_{i \in I} y_i z_i$, and interpret $\int_0^\cdot \eta \ud P \equiv \int_0^\cdot \inner{\eta}{\ud P}_{\bX} = \sum_{i \in I} \eta_i \ud P_i$. Note that $\bX$ is a strict subset of $\bY$, that the inner product $\inner{\cdot}{\cdot}_{\bX}$ endows $\bX$ with a strictly stronger topology than the one inherited from $\inner{\cdot}{\cdot}_{\bY}$, and that almost every path of $P$ lives outside of $\bX$, since $\sum_{i \in I} |P_i(t)|^2 = \infty$ holds for all $t > 0$. Importantly, and as has been mentioned already, the Hilbert space $(\bX, \inner{\cdot}{\cdot}_\bX)$ does not depend on the choice of $\bY$, i.e., on the chosen weights $(b_i; \, i \in I)$. There is no actual purpose of initially restricting $P$ to take values in $\bY$; one could carry out the above program without any reference to $\bY$, and construct $\bX$ intrinsically. Indeed, all that is required to ensure that $\sum_{i \in I} \eta_i \ud P_i$ is formally well defined is that the putative quadratic variation process $\int_0^\cdot \sum_{(i, j) \in I \times I} \eta_i \ud P_i \ud P_j \eta_j = \int_0^\cdot \norm{\eta(t)}^2_\bX \ud t$ is finite, for which \emph{only information on the local covariation structure of $P$ is necessary}.

\subsection*{Contribution}

This work aims at extending the points of the last paragraph above in the context of continuous semimartingales $P \equiv (P_i; \, i \in I)$. Stochastic integration is approached in an \emph{agnostic} way, imposing no assumptions regarding the structure of the index set $I$, and with no \emph{a priori} restrictions on the subspace of $\bR^I$ that $P$ may be taking values. For local martingale integrators $P$, we construct a topological bijection of the $\Sem$-closed space $\Sem(P)$ of ``extended stochastic integrals'' with an appropriate space $\RKH(C)$ of integrands. The latter is a dynamic version of reproducing kernel Hilbert space (rkHs) with respect to the stochastic aggregate kernel $C \equiv (C_{ij}; \, (i, j) \in I \times I)$, consisting of the processes $C_{ij} \dfn [P_i, P_j]$ of aggregate covariations between $P_i$ and $P_j$ for $(i, j) \in I \times I$. This bijection is then extended to semimartingale integrators with the structural property that the collection of finite variation drift processes of $P$ belongs in the space $\RKH(C)$.

In order to have a preview of how this program is carried out, let us revisit the case of a finite index set $I$ and a family $P$ of continuous local martingales. We identify the appropriate local\footnote{By ``local'' here and below we mean dependent on $(\omega, t)$ in the product space $\Omega \times \bR_+$ of scenarios in $\Omega$ and time in $\bR_+$, where stochastic processes are defined.} geometry of $\bR^I$, tailored for extension in infinite dimensional stochastic integration. As previously, and in order to keep things on an intuitive level, we work with formal differential quantities. The local covariation matrix $\ud C$ of $\ud P$, regarded as a kernel on $I \times I$, induces the local rkHs  $\rkh(\ud C) = \{ (\ud C) \eta \such \eta \in \bR^I \}$ (the image of $\ud C$) with inner product satisfying $\inner{\gamma}{\delta}_{\ud C} = \sum_{i \in I} \eta_i \delta_i$ whenever $\gamma \equiv (\ud C) \eta$ and $\delta$ are elements of $\rkh(\ud C)$. Given $X \equiv \int_0^\cdot \sum_{i \in I} h_i \ud P_i$, where $h$ is predictable and $P$-integrable, let $F = ([X, P_i]; \, i \in I)$ be the aggregate covariation processes of $X$ with respect to $P$. Then, $\ud F = (\ud C) h$; therefore, $\ud X = \sum_{i \in I} h_i \ud P_i = \inner{\ud F}{\ud P}_{\ud C}$. Furthermore, $\norm{\ud F}^2_{\ud C} = \sum_{i \in I} h_i \ud C_{i j} h_j = \ud [X, X]$ holds for the quadratic variation $[X, X]$ of $X$. Straightforward reverse engineering shows that we may characterise the class $\RKH(C)$ of integrands as collections $F \equiv (F_i; i \in I)$ of finite variation processes for which the putative quadratic variation process $\int_0^\cdot \norm{\ud F}^2_{\ud C}$ is  finitely valued. It is exactly for such $F \in \RKH(C)$ that the stochastic integral $X^F = \int_0^\cdot \inner{\ud F}{\ud P}_{\ud C}$ is well defined and satisfies the It\^{o} isometry $[X^F, X^F] = \int_0^\cdot \norm{\ud F}^2_{\ud C}$.

Carefully ironing out details, the above discussion extends when $I$ is an arbitrary index set. One starts with integrals of increments $\inner{\ud F}{\ud Z}_{\ud C}$ for aggregate covariation processes $F \equiv (F_i; \, i \in I)$ that involve integration only with finite subsets $J \subseteq I$ of integrators, with the remaining ``coordinates'' $(F_i; \, i \in I \setminus J)$ being completely specified. Then, via suitable natural approximation, the general stochastic integral is defined. More precisely:

\begin{enumerate}
	\item First, the space $\RKH(C)$ is constructed, using as only input a stochastic aggregate kernel $C$ as in Definition \ref{def:stoch_aggr_kernel}. In order to ensure that measurability issues are avoided, the construction is made ``from the ground up'', inspired by the way general rkHs can be defined via approximations from finite-dimensional ones, and not abstractly as completions of pre-Hilbert spaces. This is carried out in full detail in Section \ref{sec:stoch_rkhs}, with certain prerequisites on usual rkHs given in Appendix \ref{appsec:rkhs}.
	\item Secondly, in the case of continuous local martingales $P = (P_i ; \, i \in I)$, and with $C$ generated by $P$ cia $C_{ij} = [P_i, P_j]$, $(i, j) \in I \times I$, we establish a bijection (and a topological isomorphism) of the spaces $\RKH (C)$ and $\Sem(P)$ via the mapping $\Sem (P) \ni X \mapsto ([X, P_i]; i \in I) \in \RKH(C)$. This material constitutes the first half of Section \ref{sec:stoch_int}.
	\item Thirdly, we investigate the extent to which the mapping $\Sem (P) \ni X \mapsto ([X, P_i]; i \in I) \in \RKH(C)$ forms a bijection between $\RKH(C)$ and $\Sem(P)$ when $P$ is a collection of continuous semimartingales, with Doob-Meyer decompositions $P_i = A_i + M_i$, $i \in I$, where $A \equiv (A_i; \, i \in I)$ are continuous processes of finite variation, and $M \equiv (M_i; \, i \in I)$ are continuous local martingales. The main insight is that the \emph{structural condition} $A \in \RKH(C)$ is both necessary and sufficient for $\Sem (P) \ni X \mapsto ([X, P_i]; i \in I) \in \RKH(C)$ to be a bijection; and a complete operational characterisation of $\Sem(P)$ is possible. This is done in the second half of Section \ref{sec:stoch_int}, culminating with Theorem \ref{thm:struct_rkhs}.
\end{enumerate}

The above approach has pedagogical benefits: it does not require\footnote{It should be noted, however, that the theory of Banach-valued stochastic processes is both elegant and powerful, offering a more in-depth understanding of stochastic analysis, even though the present approach regarding stochastic integration does not strictly \emph{require} this knowledge.} prior knowledge of infinite-dimensional stochastic analysis, and constructs stochastic integrals via natural approximation using the well-understood finite-dimensional integration theory. The local geometry used on $\bR^I$ is the closest relative to the one in finite-dimensional Euclidean space: rkHs are endowed with an inner product structure leading to a topology where evaluation functionals $\bR^I \ni x \mapsto x_i \in \bR^I$ are continuous for every $i \in I$, and intuition gathered from finite-dimensional Euclidean spaces typically carries through without pitfalls. 

\subsection*{Applications to Mathematical Finance}

Models with an infinite number of assets have been considered extensively in the field of Mathematical Finance, often dealing with questions of (absence of) arbitrage, completeness, and optimisation. In the context of so called \emph{large financial markets}, there is work at the pre-limit in \cite{MR1348197, MR1806101} to study absence of arbitrage, as well in the post-limit in \cite{MR2178505}, where hedging and utility maximisation in models with countable infinity of assets is discussed, and the $\Sem$-topology plays a prominent role. The theoretical modelling of fixed-income markets involves a continuum of zero-coupon bonds, indexed by their maturities. In \cite{HJM:92}, the martingale property of discounted bond prices was characterised though a condition that explicitly connects the drift and covariance structure of forward rates. In \cite{MR2976683}, a version of trading in bond markets was proposed, using measure-valued integrands in order to accommodate for the continuum of maturities; even so, the resulting class of integrals may not be $\Sem$-closed, and concepts such as \emph{approximate} completeness are used to circumvent the fact.  Despite these efforts, there has not been a unifying treatment of models with arbitrary number of assets that is as satisfactory as the theory in the finite-asset case; a notable exception is \cite{MR3583448}, containing a more abstract treatment of markets with an infinity of assets, where the importance of the $\Sem$-topology is re-enforced, but without a concrete operational  characterisation of the class of wealth processes.

Using the present construction of stochastic integrals, the cornerstone results of the theory of Mathematical Finance carry \emph{mutatis mutandis}. To begin with, the structural condition that the collection $A$ of finite variation drift processes of $P$ belongs in the space $\RKH(C)$, that allows one to characterise stochastic integrals in terms of integrators $\RKH(C)$, is the exact necessary and sufficient condition to ensure (a version of) market viability. This viability condition has had several incarnations in previous literature as \emph{no arbitrage of the first kind} in \cite{MR1348197}, \emph{condition ``BK''} in \cite{MR1647282}, \emph{No Unbounded Profit with Bounded Risk} in \cite{MR2335830}. It is weaker than the condition of \emph{No Free Lunch with Vanishing Risk} in \cite{MR1304434}; it is is fact the weakest notion such that, together with the $\Sem$-closedness of the class $\Sem(P)$ of stochastic integrals, allows other fundamental results such as the optional decomposition theorem and hedging duality, to be proved. These results, in turn, allow to apply abstract results of \cite{MR1722287, MR2023886} and \cite{MR3292127} to solve utility maximisation problems.

We demonstrate in Section \ref{sec:math_fin} how to carry out this program and prove the \emph{fundamental} Theorem \ref{thm:ftap}, connecting market viability, existence of local martingale deflators and the structural condition $A \in \RKH(C)$, the \emph{optional decomposition} Theorem \ref{thm:odt} and its consequence, the \emph{hedging duality} Theorem \ref{thm:hedging}, as well as the \emph{second fundamental} Theorem \ref{thm:ftap2} involving completeness. In \S \ref{subsec:HJM}, we give an example of how the theory is applied by specialising to the context of Heath-Jarrow-Morton bond markets.

We only consider here continuous-path asset prices, as the theory of infinite-asset markets becomes more delicate when jumps may appear. Indeed, an illuminating example in \cite[Section 6]{MR3583448} shows, even the strong condition of \emph{No Free Lunch with Vanishing Risk} can only ensure existence of supermartingale (but not necessarily local martingale) deflators in the market. Contrary to the finite-asset case as in \cite{MR1304434}, \cite{MR3177411} and \cite{MR3551861}, one cannot expect an analogue of the fundamental Theorem \ref{thm:ftap} to hold.

\subsection*{Notation} 

Time will be evolving continuously in $\bR_+ \equiv [0, \infty)$. All stochastic elements will be defined on a filtered probability space $\basis$, where $\cF (\cdot) = \pare{\cF (t); t \in \bR_+}$ is a right-continuous filtration and $\bP$ a probability on $(\Omega, \cF)$, where $\cF \equiv \bigvee_{t \in \bR_+} \cF(t)$. Unless otherwise explicitly mentioned, all relationships between random variables are understood to hold in the $\bP$-a.e. sense, and all relationships between stochastic processes are understood to hold outside a $\bP$-evanescent set.

We denote by $\FV$ the set of  all adapted and right-continuous scalar processes $B$ of finite first variation on compact time intervals, with $B(0) = 0$. Furthermore, $\Mloc$ will denote the set of all local martingales on $\basis$. The set $\Sem$ consists of all semimartingales on $\basis$, that is, processes that can be decomposed as sums of elements from $\FV$ and $\Mloc$. The qualifier ``$\co$'' in front of the previous sets (as in $\cFV$, $\cMloc$ and $\cSem$) denotes the corresponding subset that consists of processes with continuous paths.

For arbitrary nonempty index set $I$, we write $\Fin (I)$ (respectively, $\Cou (I)$) for the collection of all non-empty subsets of $I$ with finite (respectively, at most countably infinite) cardinality. Whenever $\class$ is a given set of processes, $\class^I$ will denote the collection of processes of the form $D \equiv (D_i; \, i \in I)$ with $D_i \in \class$ for all $i \in I$. We  stress that elements of $\class^I$ are regarded simply as collections of scalar processes from $\class$, and \emph{not} as $\bR^I$-valued processes. This point of view sheds away potential measurability issues that would result from aggregating uncountably many processes into a single one, without assuming any structure on the index set $I$.

\section{Stochastic Aggregate Reproducing Kernel Hilbert Space} \label{sec:stoch_rkhs}

This following notion is central to the whole Section.

\begin{definition}
\label{def:stoch_aggr_kernel}
A collection $C \equiv \pare{C_{i j}; \, (i, j) \in I \times I} \in \cFV^{I \times I}$ of adapted, continuous processes of finite variation will be called an \textbf{stochastic aggregate kernel on $I \times I$}, if, for each fixed pair  $(i, j) \in I \times I$, $C_{i j} = C_{j i}$ holds, as well as
	\begin{equation} \label{eq:diff_covar_pos_def}
	\sum_{(i, j) \in J \times J} z_i \pare{C_{ij} (t) - C_{ij} (s)} z_j \geq 0, \quad \text{for } 0 \leq s \leq t, \ J \in \Fin(I), \text{ and } (z_i; \, i \in J) \in \bR^J.
	\end{equation}
\end{definition}

The properties of such a stochastic aggregate kernel  $C$ can be formally described via  the requirement that the ``differential'' process $\ud C$ takes values in the collection of kernels on $I \times I$,   defined at the start of Appendix \ref{appsec:rkhs}.  However, a certain technical issue arises already,   from the possibility that the index set $I$ might be uncountable. For every \emph{fixed} pair $(i, j) \in I \times I$, the processes $C_{i j}$ and $C_{j i}$ have continuous paths  of finite variation, and the process-equality  $C_{i j} = C_{j i}$ holds outside   an evanescent set  \emph{which may depend on $(i, j) \in I \times I$}. We do \emph{not} insist that this process-equality should  hold simultaneously for all $(i, j) \in I \times I$; while such equality is  possible  for (an at most) countable $I$, it is too much to ask for, and unnecessary for our purposes when $I$ is uncountable. The same goes for positive-definiteness: for fixed   $J \in \Fin(I)$, one may alter the processes $\pare{C_{ij};  \, (i, j) \in J \times J}$ on an evanescent set, and obtain  \eqref{eq:diff_covar_pos_def} simultaneously for all $(z_i; i \in J) \in \bR^J$ and $0 \leq s \leq t$; but it would be impossible in general to have these inequalities valid simultaneously for all finite subsets $J \in \Fin(I)$.

The canonical examples of stochastic aggregate kernels to keep in mind throughout, are those    generated by a collection $P \equiv (P_i; i \in I)$ of continuous semimartingales, via 
\begin{equation} 
\label{C:semimart}
C_{i j} \dfn \bra{P_i, P_j}, \quad (i, j) \in I \times I.
\end{equation}

\subsection{Stochastic aggregate rkHs: the finite-index set case} \label{subsec:stoch_rkhs_fin}

For the purposes of \S\ref{subsec:stoch_rkhs_fin}, we assume that the set $I$ has finite cardinality. We follow similar notational conventions as in Section \ref{appsec:rkhs} of the Appendix, and set
\[
C_{I j} \dfn (C_{ij}; \, i \in I) \in \cFV^I, \quad j \in I.
\]

We define the \textbf{stochastic aggregate rkHs} $\RKH(C)$ associated with a given stochastic aggregate kernel $C$ as in Definition \ref{def:stoch_aggr_kernel}, as the collection of all processes $F \equiv  (F_i; i \in I) \in \cFV^I$ of the form
\begin{equation} 
\label{eq:rkhs_implicit_fin}
F = \int_0^\cdot \sum_{j \in I} \theta_j (t) \ud C_{I j} (t), \quad \text{i.e.,} \quad F_i = \int_0^\cdot \sum_{j \in I} \theta_j (t) \ud C_{i j} (t), \ \ i \in I,
\end{equation}
 for a predictable process $\theta \equiv (\theta_i; \, i \in I)$   satisfying  the  
integrability condition 
\begin{equation} 
\label{eq:rkhs_stoch_norm_fin}
\int_0^T \norm{\ud F (t)}^2_{\ud C(t)} \dfn \int_0^T \sum_{(i, j) \in I \times I} \theta_i(t) \ud C_{i j} (t) \theta_j(t) < \infty, \  \ \forall \, \,T \in \bR_+ .
\end{equation}
 This condition \eqref{eq:rkhs_stoch_norm_fin} implies, in particular,   that $F$ in \eqref{eq:rkhs_implicit_fin} is well defined; because,  for all $T \in \bR_+$ and  $i \in I$, the Cauchy-Schwarz inequality gives
\[
\int_0^T \abs{ \sum_{j \in I} \theta_j (t) \ud C_{i j} (t) } \leq \sqrt{C_{ii} (T) \int_0^T \norm{\ud F (t)}^2_{\ud C(t)} } < \infty.
\]

In order to appreciate the definition of $\int_0^\cdot \norm{\ud F (t)}^2_{\ud C(t)}$ in \eqref{eq:rkhs_stoch_norm_fin}, let us  note that \eqref{eq:rkhs_implicit_fin}  reads formally $\ud F = \sum_{j \in I} \theta_j \ud C_{I j}.$  In view of the notation in   \S \ref{subsec:rkh_fin}, this  may be re-written formally as $\ud F \in \rkh (\ud C)$ and lead, formally once again, to
\[
\norm{\ud F}^2_{\ud C} = \sum_{i \in I} \theta_i \ud F_i = \sum_{(i, j) \in I \times I} \theta_i \ud C_{i j} \theta_j,
\]
the differential version of the notation in \eqref{eq:rkhs_stoch_norm_fin}.

\begin{remark}[A  description in terms of rates]
\label{rem:oper_clock_fin}
The equations \eqref{eq:rkhs_implicit_fin}, \eqref{eq:rkhs_stoch_norm_fin} can be  written more rigorously  in terms of kernel \emph{rates}, putting the above formal considerations   on solid ground. We shall explain    what this entails presently.
	
Define the continuous nondecreasing $\clo \dfn \sum_{i \in I} C_{ii}$. (Since $I$ is here assumed to have  finite cardinality, $\clo$ is finitely-valued.) Then, there exists predictable $c : \Omega \times \bR_+ \to \bR^{I \times I}$ such that
\[
C_{ij} = \int_0^\cdot c_{ij} (t) \ud \clo(t), \quad \forall \, (i, j) \in I \times I.
\]
Note that $c(\omega, t)$ is a positive-definite kernel on $I$ on a predictable set of full $\ppclo$-measure; setting $c \equiv 0$ on the complement of the previous predictable set, we may, and shall, assume that $c(\omega, t)$ is a positive-definite kernel on $I$ for \emph{every} $(\omega, t) \in \Omega \times \bR_+$.

With the above notation, the integrability condition of \eqref{eq:rkhs_stoch_norm_fin} reads
\[
\int_0^T \pare{\sum_{(i, j) \in I \times I} \theta_i(t)  c_{i j} (t) \theta_j(t) } \ud \clo (t) < \infty, \quad \forall \, \,\,T \in \bR_+.
\]
Furthermore, with $c_{I j} = (c_{ij}; \, i \in I)$ for $j \in I$, and defining the predictable $\bR^I$-valued process $f \dfn \sum_{j \in I} \theta_j c_{I j}$ by analogy with \eqref{eq:repr_ker_bil_form},  we    write concisely the process  considered in \eqref{eq:rkhs_implicit_fin}   as $F = \int_0^\cdot f (t) \ud \clo (t)$, and note  $\norm{f}^2_c = \sum_{(i, j) \in I \times I} \theta_i c_{i j} \theta_j$. Formally once again, we express this equality  as $\norm{\ud F}^2_{\ud C} = \sum_{(i, j) \in I \times I} \theta_i \ud C_{i j} \theta_j = \norm{f}^2_c \ud \clo$. In view of all this,  the process $\int_0^\cdot \norm{\ud F (t)}^2_{\ud C(t)} \in \cFV$ of \eqref{eq:rkhs_stoch_norm_fin} becomes 
\[
\int_0^\cdot \norm{\ud F (t)}^2_{\ud C(t)} \equiv \int_0^\cdot \sum_{(i, j) \in I \times I} \theta_i(t) \ud C_{i j} (t) \theta_j(t) = \int_0^\cdot \norm{f(t)}^2_{c (t)} \ud \clo (t).
\]

The above notation is more rigorous, 	but also quite a bit   more involved, than the compact  and suggestive  one in  \eqref{eq:rkhs_implicit_fin},   \eqref{eq:rkhs_stoch_norm_fin}; we shall stick with that simpler notation for the remainder of this Section. Let us also note that, when $I$ is potentially (uncountably) infinite, such a universal dominating process $\clo$ may not even exist; we shall instead  use then ideas from Lemma \ref{lem:rkhs_norm_charact} of the Appendix, in order to define the stochastic aggregate rkHs  $\RKH(C)$   in \eqref{eq:FV2C}.
\end{remark}

With $F \in \RKH(C)$ as  in \eqref{eq:rkhs_implicit_fin}  and $H=   (H_i; i \in I)   \in \RKH(C)$ with $H_i = \int_0^\cdot \sum_{j \in I} \eta_j (t) \ud C_{i j} (t)$, $i \in I$, we also introduce the process
\begin{align} 
\label{eq:polarisation}
\int_0^\cdot \inner{\ud F(t)}{\ud H(t)}_{\ud C(t)} \dfn \int_0^\cdot \sum_{(i, j) \in I \times I} \theta_i(t) \ud C_{i j} (t) \eta_j(t),
\end{align}
and note that $\int_0^\cdot \norm{\ud F(t)}^2_{\ud C(t)} = \int_0^\cdot \inner{\ud F(t)}{\ud F(t)}_{\ud C(t)}$ for $F \in \RKH(C)$ in the manner of \eqref{eq:rkhs_stoch_norm_fin}.

By definition, $\int_0^\cdot \norm{\ud C_{I j} (t)}^2_{\ud C(t)} = C_{jj}$, so that $C_{I j} \in \RKH(C)$ holds for each $j \in I$; furthermore, it is straightforward to verify the identity
\begin{equation} \label{eq:reproducing_stochastic}
F_j = \int_0^\cdot \inner{\ud C_{I j}(t)}{\ud F(t)}_{\ud C(t)}, \quad F \in \RKH(C), \ j \in I.
\end{equation}
This is the stochastic aggregate version of the reproducing kernel property in the Appendix.

\subsection{An alternative representation for the finite-index case}
\label{subsec:stoch_rkhs_fin_alt}
 
We continue assuming that $I$ is a nonempty index set of finite cardinality.

Just as in Remark \ref{rem:rkhs_norm_charact} of  \S \ref{subsec:rkh_alt}, here also there is  an alternative representation for  the  stochastic aggregate rkHs  $\RKH(C)$ of processes in \eqref{eq:rkhs_implicit_fin}, \eqref{eq:rkhs_stoch_norm_fin}. To wit, we shall associate with \emph{every given} $F \equiv (F_i; i \in I) \in \cFV^I$ a nondecreasing process  $  \int_0^\cdot \norm{\ud F(t)}^2_{\ud C(t)}    $;   then $\RKH(C)$  is   the collection of all such processes $F \in \cFV^I$, for which   $\int_0^\cdot \norm{\ud F(t)}^2_{\ud C(t)}$ is finitely-valued.

Formally, this is done as follows: We define by analogy with \eqref{eq:rkhs_pre_image_fin_dim} the predictable processes 
\[
\theta^{F; n} \dfn \pare{\ud C +   \frac{1}{n} 
\sum_{i \in I}  |\ud F_i| \, \id_{\bR^I}}^{-1} \ud F, \quad F \equiv (F_i; i \in I) \in \cFV^I, \ n \in \bN.
\]
The only difference with \eqref{eq:rkhs_pre_image_fin_dim}, is the multiplicative factor $\sum_{i \in I}  |\ud F_i|$ in the expression $\ud C + (1/n) \sum_{i \in I}  |\ud F_i| \, \id_{\bR^I}$;  this is there, to ensure that $\ud F$ is always in the range of the latter matrix differential.\footnote{Multiplying $\id_{\bR^I}$ in \eqref{eq:rkhs_pre_image_fin_dim} with any strictly positive constant will result in the exact same development in the static setting of Section \ref{appsec:rkhs}. In contrast, multiplication of $\id_{\bR^I}$ by $\sum_{i \in I}  |\ud F_i|$ becomes important here because of the dynamic setting we are dealing with; to wit, we need to ensure that, locally in time, $\theta^{F; n}$ is well defined, as we do not assume a priori that the components of $(F_i; \, i \in I)$ are absolutely continuous with respect to $\clo$.}
We  introduce then a nondecreasing, $[0, \infty]$-valued process $\int_0^\cdot \norm{\ud F(t)}^2_{\ud C(t)}$ via
\[
\int_0^T \norm{\ud F(t)}^2_{\ud C(t)} \equiv    \lim_{n \to \infty} \uparrow \int_0^T \inner{\theta^{F; n} (t)}{\ud F(t)}_{\bR^I}, \quad T \in \bR_+.
\]
With   this in mind, we have the identification of the  stochastic aggregate rkHs $\RKH(C)$ as 
\[
\RKH (C) \equiv \set{F \in \cFV^I \ \Big| \  \int_0^T \norm{\ud F(t)}^2_{\ud C(t)} < \infty, \ \, \forall \, \,T \in \bR_+}.
\]
Indeed, it is straightforward to check that a given process $F \in \cFV^I$  belongs to the set on the right-hand-side of the above equality  if, and only if, the condition  \eqref{eq:rkhs_implicit_fin} holds for some predictable $\theta \equiv \theta^F$ satisfying \eqref{eq:rkhs_stoch_norm_fin}. In fact, and again by analogy with Lemma \ref{lem:norm_bare_hand}, one choice for such a process is
\begin{equation} \label{eq:integrand_from_covar_fin}
\theta^F = \lim_{n \to \infty} \theta^{F; n} = \lim_{n \to \infty} \pare{\ud C + (1/n) \sum_{i \in I}  |\ud F_i| \, \id_{\bR^I}}^{-1} \ud F.
\end{equation}

We note that a process $F \in \cFV^I$  can fail to belong to the stochastic aggregate rkHs $\RKH(C)$ for a variety of reasons. First, the variation process $\int_0^\cdot \sum_{i \in I}  |\ud F_i|$ may fail to be absolutely continuous with respect to $\clo$   defined in Remark \ref{rem:oper_clock_fin}. Secondly, even if $F = \int_0^\cdot f(t) \ud \clo(t)$ holds for appropriate predictable $f \equiv (f_i; i \in I)$, it may happen that $\set{f \in \rkh(c)} = \set{\norm{f}_c < \infty}$ fails to have full $\ppclo$-measure. Finally, even when  $F = \int_0^\cdot f(t) \ud \clo(t)$ holds and $\set{\norm{f}_c < \infty}$ has full $\ppclo$-measure, it can very well be that $\norm{f}_c$ fails to be square-integrable with respect to $\clo$, $\bP$-a.e., over some  compact time-interval(s).

\subsection{A digression on nondecreasing processes} \label{subsec:digr_inc_proc}

In   \S \ref{subsec:stoch_rkhs_gen},  we shall   extend the material of \S \ref{subsec:stoch_rkhs_fin}--\ref{subsec:stoch_rkhs_fin_alt} to general index sets. We shall need along the way some facts regarding nondecreasing processes; these are presented now.

For any two nondecreasing, though not necessarily right-continuous,   processes $\Phi$ and $\Psi$ with values in $(- \infty, \infty]$, we write
\[
\Phi \preceq \Psi \quad \Longleftrightarrow \quad \Phi \leq \Psi, \text{ and } 
\Psi - \Phi \text{ is nondecreasing on } \set{ \Phi < \infty}.
\]
We denote by $\FV_{\succeq}$   the class of all processes $\Phi \in \FV$ which are nonnegative and nondecreasing, i.e., with $\Phi \succeq 0$; furthermore, $\cFV_{\succeq}$ is the class of all elements of $\FV_{\succeq}$ with continuous paths.

\begin{lemma} 
\label{lem:ess_sup_inc}
Let $(\Lambda, \leq)$ be a directed set, and $(\Phi_\lambda; \, \lambda \in \Lambda)$ be a collection of processes in $\cFV_{\succeq}$ such that $\Phi_\lambda \preceq \Phi_\mu$ holds whenever $\lambda \leq \mu$ and $\esssup_{\lambda \in \Lambda} \Phi_\lambda(T) < \infty$ holds for all $T \in \bR_+$. There exists then a process in $\cFV_{\succeq}$,   denoted by $\bigvee_{\lambda \in \Lambda} \Phi_\lambda$, such that
\[
\bigvee_{\lambda \in \Lambda} \Phi_\lambda (T) = \esssup_{\lambda \in \Lambda} \Phi_\lambda(T),  \quad \forall \, T  \in \bR_+,
\]
as well as a nondecreasing sequence $(\lambda^n; n \in \bN)$ in $\Lambda$ with
\[
\lim_{n \to \infty}  \Phi_{\lambda^n} = \bigvee_{\lambda \in \Lambda} \Phi_\lambda.
\]
Here the convergence is monotone with respect to the $\preceq$ order; in particular,    $\pare{\Phi_{\lambda^n}; \, n \in \bN }$ converges to $\bigvee_{\lambda \in \Lambda} \Phi_\lambda$ uniformly on compact time-intervals.
\end{lemma}

\begin{proof}
For all $T \in \bR_+$, define $\Psi(T) \dfn \esssup_{\lambda \in \Lambda}  \Phi_\lambda (T)$. At this point, $\pare{\Psi(T); \, T \in \bR_+}$ is simply a a collection of nonnegative random variables, without any path-continuity properties.
	
In view of the fact that $(\Lambda, \leq)$ is a directed set and  $(\Psi_\lambda; \, \lambda \in \Lambda)$ is $\preceq$-monotone,  we infer for every $T \in \bR_+$ the existence of a nondecreasing sequence $(\lambda^{T, n}; n \in \bN)$ in $\Lambda$ such that $\lim_{n \to \infty}  \Phi_{\lambda^{T,n}} (T) = \Psi(T)$, where the convergence is monotone. Since $(\Lambda, \leq)$ is a directed set, we can define inductively a nondecreasing sequence $(\lambda^n; \, n \in \bN)$ in $\Lambda$ with the property $\lambda^{k, n} \leq \lambda^n$ for all $k \in \bN$, $n \in \bN$ with $k \leq n$. Then, using again the facts that $(\Lambda, \leq)$ is a directed set and  $(\Psi_\lambda; \, \lambda \in \Lambda)$ is $\preceq$-monotone, we get $\lim_{n \to \infty}  \Phi_{\lambda^n} (k) = \Psi(k)$ for all $k \in \bN$. Since $\pare{\Phi_{\lambda^n}; \, n \in \bN}$ is $\preceq$-monotone, there exists $\widetilde{\Psi} \in \cFV$ such that $\lim_{n \to \infty}  \Phi_{\lambda^n} = \widetilde{\Psi}$, where this process-convergence is $\preceq$-monotone and, therefore, uniform on compact time-intervals. We need only  show that $\Psi(T) = \widetilde{\Psi} (T)$ holds for every $T \in \bR_+$.
	
Clearly, $\widetilde{\Psi} (T) \leq \Psi(T)$ holds for every $T \in \bR_+$, and we already know that $\widetilde{\Psi} (k) = \Psi(k)$ holds for every $k \in \bN$. Fix an arbitrary $T \in \bR_+$, and pick $k \in \bN$ with $T \leq k$. Recall that $(\lambda^{T, n}; n \in \bN)$ is a nondecreasing sequence in $\Lambda$, such that $\lim_{n \to \infty}  \Phi_{\lambda^{T,n}} (T) = \Psi(T)$. Let $(\mu^{T, n}; n \in \bN)$ be a nondecreasing sequence in $\Lambda$ such that $\lambda^{T,n} \leq \mu^{T, n}$ and $\lambda^n \leq \mu^{T, n}$ holds for all $n \in \bN$; of course, we still have $\lim_{n \to \infty}  \Phi_{\mu^{T,n}} (T) = \Psi(T)$. Since $\Phi_{\lambda^n} \preceq \Phi_{\mu^{T,n}}$, it follows that $\Phi_{\mu^{T,n}}(T) - \Phi_{\lambda^n} (T) \leq \Phi_{\mu^{T,n}}(k) - \Phi_{\lambda^n} (k)$ holds for all $n \in \bN$, and upon taking limits we obtain $\Psi(T) - \widetilde{\Psi}(T) \leq \Psi(k) - \widetilde{\Psi}(k) = 0$; this gives $\Psi(T) \leq \widetilde{\Psi}(T)$, and completes the argument. 
\end{proof}

\begin{remark} \label{rem:increasing_bddness_fail}
In the notation of the statement of Lemma \ref{lem:ess_sup_inc}, assume the existence of $T > 0$ such that $\bP \bra{\esssup_{\lambda \in \Lambda} \Phi_\lambda(T) = \infty} > 0$. Then, it is straightforward to infer the    existence of a nondecreasing sequence $(\lambda^n; \, n \in \bN)$ in $\Lambda$ such that $\bP \bra{\lim_{n \to \infty} \Phi_{\lambda^n} (T) = \infty} > 0$, where the limit inside the latter probability expression is nondecreasing.
\end{remark}

\subsection{A stochastic analogue of rkHs: the general case} \label{subsec:stoch_rkhs_gen}

As in Appendix \ref{appsec:rkhs}, for  an arbitrary given, nonempty index set $I,$  we use $\Fin (I)$ and $\Cou (I)$ to denote, respectively,  the collection of all finite and countable subsets of $I$. We fix a stochastic aggregate kernel $C\equiv \pare{C_{i j}; \, (i, j) \in I \times I} \in \cFV^{I \times I}$  as in Definition \ref{def:stoch_aggr_kernel}.

For any collection of processes $F = F_I \equiv (F_i; i \in I) \in \cFV^I,$ and  any given subset $J \subseteq I$, we let $F_J \equiv (F_i; \, i \in J)  \in \cFV^J$. The processes
\[
\int_{0}^\cdot \norm{\ud F_J (t)}^2_{\ud C_{J J} (t)}, \quad J \in \Fin (I) 
\]
are then defined as in \S \ref{subsec:stoch_rkhs_fin_alt}; in view of \eqref{eq:rkh_norm_incr},  we have formally
\begin{equation} \label{eq:int_rkh_norm_incr_formal}
J  \in \Fin(I), Q \in \Fin(I) \text{ with } J \subseteq Q \quad \Longrightarrow \quad \norm{\ud F_J}^2_{\ud C_{JJ}} \leq \norm{\ud F_Q}^2_{\ud C_{Q Q}}.
\end{equation}
Indeed, the inequality holds  because, formally once again, $\norm{\ud F_J}^2_{\ud C_{JJ}}$ 
is the squared norm of the orthogonal $\rkh(\ud C_{QQ})$-projection of $\ud F_Q$ on $\rkh (\ud C_{QQ}; J)$; we recall again the notation in    \S \ref{subsec:rkh_fin}. Working with the proper definitions of these quantities as in \S \ref{subsec:stoch_rkhs_fin}, and recalling the notation of  \S \ref{subsec:digr_inc_proc}, we obtain a rigorous and precise version of the comparison \eqref{eq:int_rkh_norm_incr_formal} as   follows:
\begin{equation} \label{eq:int_rkh_norm_incr_proper}
\int_{0}^\cdot \norm{\ud F_J (t)}^2_{\ud C_{JJ} (t)} \preceq \int_{0}^\cdot \norm{\ud F_Q (t)}^2_{\ud C_{Q Q} (t)}, \quad J  \subseteq Q \in \Fin(I).
\end{equation}

By analogy with Lemma \ref{lem:rkhs_norm_charact} and   Remark \ref{rem:rkhs_norm_charact}, we define   now the \textbf{stochastic aggregate rkHs} associated with the  given stochastic aggregate kernel $C$, as the collection of processes
\begin{equation} 
\label{eq:FV2C}
\RKH (C) \dfn \set{F \in \cFV^I \ \Big| \   \esssup_{J \in \Fin(I)} \int_{0}^T \norm{\ud F_J (t)}^2_{\ud C_{JJ} (t)} < \infty, \quad \forall \ T \in \bR_+}.
\end{equation}
This space will accommodate the cumulative covariations of the extended stochastic integrals we will construct in the next Section \ref{sec:stoch_int}  with respect to a   collection $P = (P_i ; i  \in I) $ of continuous semimartingales. The  ``internal'' covariations of these integrands, $C_{ij} = [ P_i, P_j]$  as in \eqref{C:semimart}, will be represented by the stochastic aggregate kernel $C$   of Definition \ref{def:stoch_aggr_kernel}.

We consider now an arbitrary element $F \equiv (F_i; \, i \in I)$ of the stochastic aggregate rkHs  $\RKH (C)$ just defined. In view of Lemma \ref{lem:ess_sup_inc}, and of the fact that   $\Fin(I)$ equipped with the usual set-inclusion order is an ordered set, the comparison \eqref{eq:int_rkh_norm_incr_proper} implies that    an ``essential supremum'' process can be defined via
\begin{equation} \label{eq:stoch_rkhs_general}
\int_{0}^\cdot \norm{\ud F (t)}^2_{\ud C (t)} \dfn \bigvee_{J \in \Fin(I) } \int_{0}^\cdot \norm{\ud F_J (t)}^2_{\ud C_{JJ} (t)}.
\end{equation}
Furthermore, there exists a nondecreasing sequence $(J^n; \, n \in \bN)$ in $\Fin(I)$ such that
\begin{equation} \label{eq:stoch_rkhs_approx}
\int_{0}^\cdot \norm{\ud F (t)}^2_{\ud C (t)} = \lim_{n \to \infty} \int_{0}^\cdot \norm{\ud F_{J^n} (t)}^2_{\ud C_{J^n J^n} (t)},
\end{equation}
where the last process-convergence is $\preceq$-monotone. If $I$ is at most countably infinite, then \eqref{eq:stoch_rkhs_approx} holds for \emph{any} nondecreasing sequence $(J^n; \, n \in \bN)$ in $\Fin(I)$ with $\bigcup_{n \in \bN} J^n = I$.

For any finite subset $J \in \Fin(I)$ with $j \in J$, we have the identity 
$\int_0^\cdot \norm{\ud C_{J j} (t)}^2_{\ud C_{JJ}(t)} = C_{jj}$; it follows that $C_{I j} \in \RKH(C)$ and
\[
\int_0^\cdot \norm{\ud C_{I j} (t)}^2_{\ud C(t)} = C_{jj}, \quad j \in I
\]
hold. Furthermore, for $ F \in \RKH(C)$ and $H \in \RKH(C)$, we use polarization to define 
\[
\int_0^\cdot \inner{\ud F(t)}{\ud H(t)}_{\ud C(t)} = \frac{1}{4} \left( \int_0^\cdot \norm{\ud (F + H) (t)}^2_{\ud C(t)} - \int_0^\cdot \norm{\ud (F - H) (t)}^2_{\ud C(t)} \right).
\]
A straightforward approximation argument shows that the reproducing kernel relation \eqref{eq:reproducing_stochastic} is   valid once again. Finally, for $F \in \RKH(C)$ and $H \in \RKH(C)$,  we have
\[
\int_0^\cdot \abs{\inner{\ud F(t)}{\ud H(t)}_{\ud C(t)} } \leq \sqrt{\int_0^\cdot \norm{\ud F (t)}^2_{\ud C (t)} } \sqrt{ \int_0^\cdot \norm{\ud H (t)}^2_{\ud C (t)} }.
\]

\begin{remark} 
\label{rem:stoch_rkhs_continuous}
Suppose that the index set $I$ can be endowed with a topology admitting a countable dense subset $Q$, and that there exists $\clo \in \cFV_{\succeq}$ with the property
\[
C_{ij} = \int_0^\cdot c_{ij} (t) \ud \clo(t), \quad \forall ~ (i, j) \in I \times I.
\]
Here $c : (\Omega \times \bR_+) \times (I \times I) \to \bR$ is a $(\cP \otimes \cB (I \times I))$-measurable random field  such that, for $\ppclo$-a.e. $(\omega, t) \in \Omega \times \bR_+$, $c(\omega, t)$ is a kernel on $I \times I$ and has the following properties:
\begin{itemize}
	\item $(c_{ij}(\omega, t); i \in I) \in \rkh(c(\omega, t))$ is continuous in the topology of $I$, for every $j \in I$;
	\item for every $i \in I$, there exists an open set $J(\omega, t, i) \subseteq I$ with $\sup_{j \in J (\omega, t, i)} c_{jj}(\omega, t) < \infty$.
\end{itemize}
(For example, note that these properties always hold, when $I$ is at most countable and endowed with the discrete topology. In \S \ref{subsec:HJM}, we shall see an example with uncountable $I$.)

Then, consulting Remark \ref{rem:rkhs_continuous_dense} in the Appendix, it is straightforward to check that a given $F \equiv (F_i; \, i \in I) \in \cFV^I$ belongs to the stochastic aggregate rkHs $\RKH(C)$ if, and only if,  the representation  $F = \int_0^\cdot f(t) \ud \clo (t)$ holds for some $(\cP \otimes \cB (I))$-measurable-measurable $f :  (\Omega \times \bR_+) \times I \to \bR$ with $f(\omega, t) \in \rkh (c(\omega,t))$ for $\ppclo$-a.e. $(\omega, t) \in \Omega \times \bR_+$, and
\[
\int_0^T \norm{f(t)}_{c(t)}^2 \ud \clo (t) < \infty, \quad \bP \text{-a.e.}, \quad \forall \, T \in \bR_+. 
\]
\end{remark}

\begin{remark}[Independent Brownian case]
\label{rem:Brownian_rkHs}
Let $I$ an arbitrary index set. For any $J \in \Cou(I)$, define $\ell_J^2$ as the Hilbert space consisting of all $y \equiv (y_j; j \in J)$ with the property $\sum_{j \in J} |y_j|^2 < \infty$. We equip this space $\ell_J^2$ with the inner product $\inner{\cdot}{\cdot}_{\ell^2_J}$ defined via 
\[
\inner{y}{z}_{\ell_J^2} = \sum_{j \in J} y_j z_j, \qquad y = (y_j; j \in J) \in \ell_J^2, \quad z = (z_j; j \in J) \in \ell_J^2.
\]

Suppose now that $C_{ii}(t) = t$, $t \in \bR_+$, holds for all $i \in I$, and $C_{ij} \equiv 0$ whenever $I \ni i \neq j \in I$. This specification corresponds to a continuous positive-definite stochastic  kernel on $I$ generated by a collection of independent Brownian motions. In this context, it is straightforward to see that $F \equiv (F_i; \, i \in I) \in \RKH(C)$ if,  and only if, there exists $J \in \Cou(I)$ such that
\begin{itemize}
	\item $F_i \equiv 0$ for $i \in I \setminus J$;
	\item there exists a family $f_J \equiv (f_j; \, j \in J)$ of predictable processes with $\int_0^T \norm{f_J (t)}^2_{\ell^2_J} \ud t < \infty$ for all $T \in \bR_+$, and $F_j = \int_0^\cdot f_j(t) \ud t$ for all $j \in J$. 
\end{itemize}

The significance  of the spaces $\RKH(C)$ as in \eqref{eq:FV2C}, when the stochastic aggregate kernel $C$ has more complicated structure than the ``independent Brownian'' one just described, is to replace the local Euclidean geometry of the $\ell^2$ spaces with the rkHs structure of the kernel represented via $\ud C$.
\end{remark}

\subsection{Restrictions and projections}

The spaces $\RKH (C_{JJ})$ for $J \subseteq I$ are defined simply by considering restrictions of elements on $J$. Then, similarly to Lemma \ref{lem:rkhs_norm_charact} and Remark \ref{rem:rkhs_norm_charact}, $F \in \RKH (C)$ holds if and only if $F_J \in \RKH (C_{JJ})$ holds for all $J \in \Cou(I)$. In this case, there exists $Q \equiv Q(F) \in \Cou(I)$ such that the process-equality
\[
\int_{0}^\cdot \norm{\ud F (t)}^2_{\ud C (t)} = \int_{0}^\cdot \norm{\ud F_Q (t)}^2_{\ud C_{Q Q} (t)}
\]
is valid. Indeed, in the notation of \eqref{eq:stoch_rkhs_approx}, $Q = \bigcup_{n \in \bN} J^n$.

For $J \in \Fin(I)$, the mapping
\[
\RKH(C_{JJ}) \ni \int_0^\cdot \sum_{j \in J} \theta_j (t) \ud C_{J j} (t) \mapsto \int_0^\cdot \sum_{j \in J} \theta_j (t) \ud C_{I j} (t) \in \RKH(C)
\]
is injective, and we call $\RKH(C; J)$ its image. This way, $\RKH(C_{JJ})$ is isometric to $\RKH(C; J)$; the inverse of the previous mapping is simply $\RKH(C; J) \ni F \mapsto F_J \in \RKH(C_{JJ})$.

\section{Stochastic Integration for Arbitrary Collections of Continuous Semimartingales} \label{sec:stoch_int}

\subsection{Continuous-semimartingale topology} \label{subsec:semimart_top}
On the set $\FV$ of adapted and right-continuous scalar processes of finite first variation on compact time intervals, we define the subadditive functional $\dnorm{\cdot}_{\FV}  : \FV \to [0,1]$  via
\begin{equation} 
\label{eq:metric_TV}
\dnorm{B}_{\FV} \dfn \sum_{k \in \bN} 2^{-k} \, \bE^\bP \bra{1 \wedge \int_0^k \abs{\ud B(t)}}, \quad B \in \FV.
\end{equation}
We consider also the topology generated by the translation-invariant metric $\FV \times \FV \ni (A, B) \mapsto \dnorm{B - A}_{\FV}$. Convergence in this topology amounts to  convergence  in probability of the  total variation, over compact  intervals.

Recall that $\cSem$ denotes the class of all continuous, scalar  semimartingales $X\equiv B + L$ with $X(0) = 0$. Here  $X \equiv B + L$ expresses the Doob-Meyer decomposition of $X$,  as the sum of $B \in \cFV$ and $L \in \cMloc$. We  introduce a subadditive functional $\dnorm{\cdot}_{\cSem} : \cSem \to [0,1]$ via 
\[
\dnorm{X}_{\cSem} \dfn \dnorm{B}_{\FV} + \dnorm{ [L, L]^{1/2} }_{\FV}, \qquad X \equiv B + L \in \cSem.
\]
The $\cSem$-topology, generated by the translation-invariant metric $\cSem \times \cSem \ni (X, Z)    \mapsto \dnorm{Z - X}_{\cSem}$ can be seen to coincide with the (localised version of the so-called) semimartingale topology of \cite{MR544800}, restricted to continuous semimartingales.

\smallskip
We shall fix from now onwards a collection $P \equiv (P_i; \, i \in I) \in \cSem^I$ and write
\[
P_i = A_i + M_i ,  \quad i \in I ,
\]
where $A \equiv (A_i; \, i \in I) \in \cFV^I$ and $M \equiv (M_i; \, i \in I) \in \cMloci$. We then define $C \equiv (C_{ij}; \, (i, j) \in I \times I) \in \cFV^{I \times I}$ via $C_{i j} \dfn [P_i, P_j]  = [M_i, M_j]$, $ (i, j) \in I \times I$ as in \eqref{C:semimart}.

For $P \equiv (P_i; \, i \in I) \in \cSem^I$, we shall denote by $\Sem(P)$  the $\cSem$-closure of the set of all stochastic integrals   which can be formed  using only a finite number of components of $P$ as integrators, and via use of simple predictable integrands. When $I$ has finite cardinality, $\Sem(P)$ coincides with the collection of all stochastic integrals that can be formed using $P$ as integrator, via use of vector stochastic integration; see \cite{MR544800} for a proof of this last  claim.

\subsection{Roadmap}

In order to set the stage and introduce some of the main actors, let us offer a bit of a preview of what is to come. We shall eventually establish in \S \ref{subsec:isomorph_semimarts}   necessary and sufficient conditions, under which    a bijection exists between the space $\Sem(P)$ on the one hand, and the  stochastic aggregate rkHs   $\RKH (C)$ of \eqref{eq:FV2C} on the other. The first of these spaces accommodates the ``extended stochastic integrals'' with respect to the collection  of continuous semimartingales $P \equiv (P_i; \, i \in I)$, whereas the second space will accommodate the ``admissible extended integrands'' of this theory, namely, the cumulative covariations of these extended stochastic integrals with the integrators $(P_i; \, i \in I)$.  The exact structural condition needed  for this bijection, appears in Theorem \ref{thm:struct_rkhs}: \emph{The drift process $A \equiv (A_i; \, i \in I)$ of $P$ has to belong to the stochastic aggregate rkHs $\RKH(C)$.} We shall see also that,   when  such a bijection exists and $\RKH (C)$  is equipped with the metric $\RKH (C) \times \RKH (C) \ni (F, H) \mapsto \dnorm{H - F}_{\RKH(C)}$ for 
\begin{equation} \label{eq:metric_VC}
\dnorm{F}_{\RKH(C)} \dfn  \dnorm{ \left( \int_{0}^\cdot  \norm{\ud F (t)}^2_{\ud C (t)} \right)^{1/2}}_{\FV}, \quad F \in \RKH(C)
\end{equation}
as in  \eqref{eq:stoch_rkhs_general} and  \eqref{eq:metric_TV}, the space $\Sem (P)$ is in fact \emph{isomorphic} to $\RKH (C)$.  This feature will allow  us   to characterize in Proposition  \ref{prop:semimart_closure_charact} the  set $\Sem(P)$  purely in terms of cumulative covariations.

\subsection{Isomorphism  for continuous local martingales} \label{subsec:isometry_loc_marts}

In preparation for obtaining the exact structural conditions, under which the space  $\Sem (P)$ of extended stochastic integrals and the  stochastic aggregate rkHs   $\RKH (C)$ of\eqref{eq:FV2C} are isomorphic to each other, we consider first the case where $P \equiv (P_i; \, i \in I) \in \cSem^I$ is a collection of continuous local martingales. We shall then use the alternative, more suggestive  notation $M \equiv (M_i; \, i \in I)$ instead of $P$, and take $M_i(0) = 0$ for all $i \in I$ without loss of generality. Recall  the stochastic aggregate kernel $C \equiv (C_{ij}; \, (i, j) \in I \times I)$   defined via $C_{ij} \dfn [M_i, M_j]$ for $(i, j) \in I \times I$ as in \eqref{C:semimart}.

\subsubsection*{Index sets of finite cardinality}

We start by assuming that $I$ is a nonempty finite set. As we have noted, $\Sem(M)$ coincides then with the collection of all local martingales that start  from zero and are stochastic integrals with respect to $M$. Consider then an arbitrary $L \in \Sem(M)$, and write $L = \int_0^\cdot \sum_{i \in I} \theta_i (t) \ud M_i(t)$, where the components of the vector process $\theta \equiv (\theta_i; \, i \in I)$ are predictable and satisfy the local integrability condition 
\begin{equation} \label{eq:integrability_classical}
\int_0^T \sum_{(i, j) \in I \times I} \theta_i(t) \ud C_{i j} (t) \theta_j(t) < \infty, \quad \forall \ T \in \bR_+.
\end{equation}
This condition is necessary for $L$ to be defined, as the quantity in \eqref{eq:integrability_classical} equals $[L, L] (T)$, which has to be finite. With $F \dfn ([L, M_i]; \, i \in I) \in \cFV^I$, the quantity 
in  \eqref{eq:integrability_classical}   equals $\int_0^T \norm{\ud F (t)}^2_{\ud C(t)}$, and consequently  $F \in \RKH(C)$. Furthermore, given any two local martingales $L \in \Sem(M)$, $N \in \Sem(M)$ with  $([L, M_i]; \, i \in I) = ([N, M_i]; \, i \in I)$, and denoting by $F \in \RKH(C)$ this common value,  we note that $[L - N, L- N] = [L, L] + [N, N] - 2 [L, N] = 0$ holds in light of the identities $[L, L] = [N, N] = [L, N] = \int_0^\cdot \norm{\ud F (t)}^2_{\ud C(t)}$. We conclude that the family $([L, M_i]; \, i \in I)$ belongs in $\RKH(C)$, and that the resulting mapping
\begin{equation} 
\label{eq:mapping_mart}
\Sem(M) \ni L \mapsto ([L, M_i]; \, i \in I) \in \RKH(C)
\end{equation}
is one-to-one.

We argue   that the mapping of \eqref{eq:mapping_mart} is also \emph{onto}, i.e., a \emph{bijection}. To see this, we fix  an arbitrary collection $F \equiv (F_i; i \in I) \in \RKH(C)$, define the predictable process $\theta^F$ as in \eqref{eq:integrand_from_covar_fin}, and note that $F_i = \int_0^\cdot \sum_{j=1}^n \theta^F_j (t) \ud C_{ij} (t)$ for $i \in I$ and
\[
\int_0^T \sum_{(i, j) \in I \times I} \theta^F_i(t) \ud C_{i j} (t) \theta^F_j(t) = \int_0^T \norm{\ud F(t)}^2_{\ud C(t)} < \infty, \quad \forall \,  T \in \bR_+.
\]
This integrability condition implies that the process
\begin{equation} \label{eq:mart_from_cov_fin}
M^F \dfn \int_0^\cdot \sum_{i \in I} \theta^F_i(t) \ud M_i(t)
\end{equation}
is a well-defined element of the space $\Sem(M)$, namely, a continuous local martingale with cross-variations given by $[M^F, M_i] = F_i$, for all $i \in I$, and with quadratic variation  
\begin{equation} 
\label{eq:isometry_fin}
\bra{M^F, M^F} = \int_0^\cdot \norm{\ud F(t)}^2_{\ud C(t)}.
\end{equation}
Since, formally, $\ud M^F = \sum_{i \in I} \theta^F_i \ud M_i = \inner{\ud F}{\ud M}_{\ud C}$, we write, suggestively,
\begin{equation} 
\label{eq:local_mart_suggestive_nota_fin}
M^F = \int_0^\cdot \inner{\ud F(t)}{\ud M(t)}_{\ud C(t)}, \quad F \in \RKH(C),
\end{equation}
for $M^F \in \Sem(M)$ in \eqref{eq:mart_from_cov_fin}.

\subsubsection*{General index sets}

We extend now the previous discussion, valid for finite index sets, to arbitrary nonempty index sets $I$. The first order of business, is again to ensure that the mapping $\Sem (M) \ni L \mapsto ([L, M_i]; \, i \in I)\in \cFV^I$ is actually $\RKH(C)$-valued. We state and prove a slightly stronger statement, for later use.

\begin{lemma} \label{lem:mapping_well_def_semimart}
For any semimartingale  $Z$, it holds that $\pare{\bra{Z, M_i}; \, i \in I} \in \RKH (C)$.
\end{lemma}

\begin{proof}
If $L$ denotes the uniquely-defined continuous local martingale part of $Z$, then $\bra{Z, M_i} = \bra{L, M_i}$ holds for all $i \in I$. Therefore, we may---and will---assume that $Z \in \cMloc$.
	
Let $F \dfn \pare{\bra{Z, M_i}; \, i \in I}$. For any given finite subset $J \in \Fin(I)$, let $N \equiv N^J$ denote the unique element of $\Sem(M_J) \subseteq \Sem(M)$ with     $F_j = \bra{ N, M_j}$   for all $j \in J$ (such $N$ exists from the Kunita-Watanabe decomposition). By the finite-index case treated previously, we have
\[
\int_0^\cdot \norm{\ud F_J (t)}^2_{\ud C_{JJ} (t)} = \bra{N, N} \leq \bra{Z, Z}.
\]
But then it follows from \eqref{eq:FV2C}, that $\int_0^T \norm{\ud F (t)}^2_{\ud C (t)} \leq \bra{Z, Z}(T) < \infty$ holds for every $T \in \bR_+$, establishing   $F \in \RKH (C)$.
\end{proof}

Lemma \ref{lem:mapping_well_def_semimart} shows that the mapping of \eqref{eq:mapping_mart} is well defined in our new context as well. We argue below that, just as in the finite-index case, this mapping  is  actually a \emph{bijection}. This result is not stated formally; it will be subsumed into the more general Theorem \ref{thm:struct_rkhs} below.

\begin{proof}[Proof of bijectivity in  \eqref{eq:mapping_mart}]
We  show first, that the mapping in \eqref{eq:mapping_mart} is one-to-one. Suppose $L \in \Sem(M)$ and $N \in \Sem(M)$ are such that $([L, M_i]; \, i \in I) = ([N, M_i]; \, i \in I)$ holds, and call $F \in \RKH(C)$ this common value. For any $J \in \Fin(I)$, let $L^J$ and $N^J$ be the unique elements in the Kunita-Watanabe decompositions on $\Sem(M_J)$ of $L$ and $N$, respectively, and note
\[
[L^J, M_j] = [L, M_j] = F_j = [N, M_j] = [N^J, M_j], \quad \forall \, j \in J.
\]
From the discussion of the finite-index-set case, it follows that $L^J = N^J$. One may then pick a common nondecreasing sequence $(J_n; \, n \in \bN)$ in $\Fin(I)$ with the property  
\[
\lim_{n \to \infty} \downarrow \big[L - L^{J_n}, \, L - L^{J_n}\big] = 0 =   \lim_{n \to \infty} \downarrow \big[N - N^{J_n}, \, N - N^{J_n}\big],
\]
from which   $L = N$ follows, showing that the mapping \eqref{eq:mapping_mart} is one-to-one.

We further argue that the mapping of \eqref{eq:mapping_mart} is also onto. We start with a given $F \in \RKH (C)$, and let $(J^n; \; n \in \bN)$ be a sequence in $\Fin(I)$ such that \eqref{eq:stoch_rkhs_approx} holds. Since $G_n \dfn F_{J^n} \in \RKH (C_{J^n J^n})$, we may define $M^{G_n}_{J^n} \in \cSem (M_{J^n})$ for all $n \in \bN$, in the notation of \eqref{eq:mart_from_cov_fin}. At this point, the process-isometries    \eqref{eq:isometry_fin},  \eqref{eq:stoch_rkhs_approx}  imply that the sequence $\big( M^{G_n}_{J^n}; \, n \in \bN \big)$ is Cauchy in $\cSem (M)$; letting $M^F \dfn \cSem$-$\lim_{n \to \infty} M^{G_n}_{J^n} \in \Sem(M)$, we obtain \eqref{eq:isometry_fin} in our present context; to wit, 
\begin{equation} 
\label{eq:isometry_gen}
\bra{M^F, M^F} = \int_0^\cdot \norm{\ud F(t)}^2_{\ud C(t)}.
\end{equation}
We claim   that $\bra{M^F, M_i} = F_i$ holds for all $i \in I$. To see this,  we fix an arbitrary index $i \in I$ and, for each $n \in \bN$, let $Q^n = J^n \cup \set{i}$. Since $H_n \dfn F_{Q^n} \in \RKH (C_{Q^n Q^n})$, it holds that $M^{H_n}_{Q^n} \in \cSem (M_{Q^n})$, for all $n \in \bN$, 
in the notation of \eqref{eq:mart_from_cov_fin}. Since \eqref{eq:stoch_rkhs_approx} holds for $(J^n; \; n \in \bN)$ and $J_n \subseteq H_n$ for all $n \in \bN$, it is immediate that $\cSem$-$\lim_{n \to \infty} M^{H_n}_{Q^n} = M^F$; and because  $i \in Q^n$, we have $\big[ M^{H_n}_{Q^n}, M_i \big] = F_i$  for all $n \in \bN$, so $[M^F, M_i] = F_i$ also holds. But $i \in I$ is arbitrary, so in fact  $([M^F, M_i]; \, i \in I) = F$,  showing that the mapping of \eqref{eq:mapping_mart} is indeed a bijection.
\end{proof}

\subsection{Stochastic integrals under structural condition for continuous semimartingales}
\label{subsec:struct_cond}

Let us  return now to a general collection $P \equiv (P_i; i \in I) \in \cSem^I$ of continuous semimartingales, and recall  their covariation structure $C \equiv (C_{ij}; (i, j) \in I \times I)$ as in   \eqref{C:semimart}.

According to Lemma \ref{lem:mapping_well_def_semimart}, the mapping $\cSem (P) \ni Z \mapsto ([Z, P_i]; \, i \in I)$ takes values in $\RKH (C)$. We saw in \S \ref{subsec:isometry_loc_marts} that the mapping
\begin{equation} 
\label{eq:mapping_sem}
\Sem(P) \ni Z \longmapsto ([Z, P_i]; \, i \in I) \in \RKH(C)
\end{equation}
is a bijection when $A \equiv 0$. Theorem \ref{thm:struct_rkhs}  below,  states that  such bijectivity is valid  under the more general  \textbf{structural condition} $A \in \RKH(C)$; and even more to the point, that this condition  is actually \emph{equivalent} to the the mapping in  \eqref{eq:mapping_sem} being bijective.

We begin with an intermediate but important structural result which provides, under the condition $A \in \RKH(C)$, a precise description for   the space $\cSem (P)$ of extended stochastic integrals. In accordance with \eqref{eq:local_mart_suggestive_nota_fin}, and from the discussion in \S \ref{subsec:isometry_loc_marts}, we set
\begin{equation} 
\label{eq:local_mart_suggestive_nota}
M^F \equiv \int_0^\cdot \inner{\ud F(t)}{\ud M(t)}_{\ud C(t)}, \quad F \in \RKH(C)
\end{equation}
for the process $M^F \in \Sem(M) \subseteq \cMloc$ that is uniquely determined by $[M^F, P_i] = [M^F, M_i] = F_i$, for all $i \in I$.

\begin{proposition}
\label{prop:semimart_closure_charact}
Under the structural condition $A \in \RKH(C)$, the space $\Sem(P)$ of extended stochastic integrals admits the representation
\begin{equation} \label{eq:semimart_closure_charact}
\Sem(P) = \set{ \int_0^\cdot \inner{\ud F(t)}{\ud A(t)}_{\ud C(t)} + \int_0^\cdot \inner{\ud F(t)}{\ud M(t)}_{\ud C(t)}  \ \Big| \ F \in \RKH(C)}.
\end{equation}
Above, we use the notation of \eqref{eq:polarisation} and  \eqref{eq:local_mart_suggestive_nota}.
\end{proposition}

\begin{proof}
Assume first that $I$ is finite. Let $Z \in \Sem(P)$ and write $Z = \int_0^\cdot \inner{\theta (t)}{\ud P(t)}_{\bR^I}$, where $\theta \equiv ( \theta_j; \, j \in I) $ has to satisfy \eqref{eq:integrability_classical}, along with $\int_0^T \abs{\inner{\theta (t)}{\ud A(t)}_{\bR^I}} < \infty$, for all $T \in \bR_+$. Setting  $F \dfn \pare{[Z, P_i]; \, i \in I} = \int_0^\cdot \sum_{j \in I} \theta_j (t) \ud C_{I j} (t)$,  and given that $A \in \RKH(C)$,  we obtain
\[
\int_0^\cdot \abs{\inner{\theta (t)}{\ud A(t)}_{\bR^I}} = \int_0^\cdot \abs{\inner{\ud F (t)}{\ud A(t)}_{\ud C(t)}} \leq \sqrt{\int_0^\cdot \norm{\ud F(t)}^2_{\ud C(t)}}  \sqrt{\int_0^\cdot \norm{\ud A(t)}^2_{\ud C(t)}},
\]
where this last process is finitely-valued. In particular, the local integrability condition \eqref{eq:integrability_classical} is necessary and sufficient for the stochastic integral $\int_0^\cdot \inner{\theta (t)}{\ud P(t)}_{\bR^I}$ to be defined;   then,
\[
Z = \int_0^\cdot \inner{\theta (t)}{\ud A(t)}_{\bR^I} + \int_0^\cdot \inner{\theta (t)}{\ud M(t)}_{\bR^I} = \int_0^\cdot \inner{\ud F (t)}{\ud A(t)}_{\ud C(t)} + \inner{\ud F (t)}{\ud M(t)}_{\ud C(t)},
\]
and \eqref{eq:semimart_closure_charact} is established.

We drop now the assumption of finite cardinality for $I$. We start by fixing $Z \in \Sem(P)$,  and set  $F \dfn ([Z, P_i]; \, i \in I)$. Consider sequences $(J^n; \, n \in \bN)$ in $\Fin(I)$ and $(Z^n; \, n \in \bN)$ in $\cSem$, such that $Z^n \in \Sem(P_{J^n})$ holds for all $n \in \bN$, and $\cSem$-$\lim_{n \to \infty} Z^n = Z$. For each $n \in \bN$, let $F^n \dfn ([Z^n, P_i]; i \in I)$. Given the Kunita-Watanabe decomposition $M^A = M^{A_{J^n}} + N^n$, with $N^n$ strongly orthogonal to the local martingales in $\Sem(M_{J^n})$  for all   $n \in \bN$, it follows that
\begin{align*}
\int_0^\cdot \inner{\ud F^n_{J^n}(t)}{\ud A_{J^n} (t)}_{\ud C_{J^n J^n}(t)} &= \bra{M^{F^n}, M^{A_{J^n}}} \\
&= \bra{M^{F^n}, M^{A}} = \int_0^\cdot \inner{\ud F^n (t)}{\ud A (t)}_{\ud C (t)}.
\end{align*}	
Thus,   the just-established result covering the   case   of finite index sets  gives
\[
Z^n = \int_0^\cdot \inner{\ud F^n (t)}{\ud A (t)}_{\ud C (t)} + \int_0^\cdot \inner{\ud F^n (t)}{\ud M (t)}_{\ud C (t)}, \quad n \in \bN.
\]
Since $\cSem$-$\lim_{n \to \infty} Z^n = Z$ implies $\bP$-$\lim_{n \to \infty} [Z^n - Z, Z^n - Z] (T) = 0$ for all $T \in \bR$, and also since $[Z^n - Z, Z^n - Z] = \int_0^\cdot \norm{\ud F^n (t) - \ud F(t)}^2_{\ud C (t)}$, we obtain  
\[
Z = \int_0^\cdot \inner{\ud F (t)}{\ud A (t)}_{\ud C (t)} + \int_0^\cdot \inner{\ud F (t)}{\ud M (t)}_{\ud C (t)}.
\]
It follows that  $\Sem(P)$ is contained in the set on the right-hand side of \eqref{eq:semimart_closure_charact}.

Conversely, start with any given $F \in \RKH(C)$. Define the subset $\RKH(C; \Fin)$ of $\RKH(C)$ consisting processes as in \eqref{eq:rkhs_implicit_fin}, but with  $I$   replaced by some arbitrary finite subset $J \in \Fin(I)$, under the suitable integrability condition as in \eqref{eq:rkhs_stoch_norm_fin}; then, $\RKH(C; \Fin)$ is dense in $\RKH(C)$ under the  metric induced by \eqref{eq:metric_VC}. Consider a sequence $(F^n; n \in \bN)$ in $\RKH(C; \Fin)$ such that $\RKH(C)$-$\lim_{n \to \infty} F^n = F$. Then, with
\[
Z^n \dfn \int_0^\cdot \inner{\ud F^n (t)}{\ud A (t)}_{\ud C (t)} + \int_0^\cdot \inner{\ud F^n (t)}{\ud M (t)}_{\ud C (t)},
\]
we have $Z^n \in \Sem(P)$ in view of the finite-index case, and
\[
\cSem \text{-}\lim_{n \to \infty} Z^n = \int_0^\cdot \inner{\ud F (t)}{\ud A (t)}_{\ud C (t)} + \int_0^\cdot \inner{\ud F (t)}{\ud M (t)}_{\ud C (t)}
\]
follows as before. Thus, the set on the right-hand side of \eqref{eq:semimart_closure_charact} is contained in $\Sem(P)$, and the proof is complete.
\end{proof}

\subsection{Isomorphism for continuous semimartingales  and   structural conditions}
\label{subsec:isomorph_semimarts}

It follows from Proposition \ref{prop:semimart_closure_charact} that, under the structural condition $A \in \RKH(C)$, the mapping of \eqref{eq:mapping_sem} is a bijection, whose inverse is given by
\[
\RKH(C) \ni F \longmapsto \int_0^\cdot \inner{\ud F (t)}{\ud A (t)}_{\ud C (t)} + \int_0^\cdot \inner{\ud F (t)}{\ud M (t)}_{\ud C (t)} \in \Sem(P).
\]

Let us recall that $\cSem$-$\lim_{n \to \infty} L^n = L^\infty$ holds for a sequence $(L^n; \, n \in \bN)$ in $\Sem(M)$   if,  and only if, we have $\lim_{n \to \infty} [L^\infty - L^n, L^\infty - L^n] (T) = 0$   for all $T \in \bR_+$. This fact, along with the process-isometry \eqref{eq:isometry_gen},  shows that the spaces $\Sem(M)$ and $\RKH(C)$ are then \emph{metrically isomorphic}, when $\RKH(C)$ is equipped with the metric  of  \eqref{eq:metric_VC}.

The next theorem generalizes these observations very considerably. Coupled with Proposition  \ref{prop:semimart_closure_charact}, it provides our main result on stochastic integration with respect to an arbitrary family, possibly uncountably-infinite, of continuous semimartingales.

\begin{theorem} \label{thm:struct_rkhs}
The following statements are equivalent:
\begin{enumerate}
	\item The mapping $\Sem (P) \ni Z \mapsto ([Z, P_i]; \, i \in I) \in \RKH (C)$ is a bijection.
	\item $A \in \RKH(C)$.
\end{enumerate}
Under these  equivalent conditions, the space $\Sem(P)$ of extended stochastic integrals admits the representation \eqref{eq:semimart_closure_charact}, and is topologically isomorphic to the stochastic aggregate rkHs $\RKH(C)$ of \eqref{eq:FV2C}.
\end{theorem}

\begin{proof}
We need  only  establish the implication $(1) \Rightarrow (2)$ and the claim regarding the topological isomorphism between $\Sem(P)$ and $\RKH(C)$; everything else has been discussed prior to the statement of Theorem \ref{thm:struct_rkhs}. Therefore, and for the remainder of this proof, we shall assume that condition (1) holds.

\smallskip
\noindent
\underline{Step 1}:
We claim that, \emph{for every  finite subset $J \in \Fin(I)$, there exists a predictable vector process $\nu^J \equiv (\nu^J_j; \, j \in J)$ such that  $A_j = \int_0^\cdot \inner{\nu^J (t)}{\ud C_{J j} (t)}_{\bR^J}$ is valid for each index  $j \in J$.} To see this, we fix $J \in \Fin(I)$  and write the decomposition 
\[
A_j = \int_0^\cdot \inner{\nu^J (t)}{\ud C_{J j} (t)}_{\bR^J} + B_j, \quad j \in J,
\]
where $B \in \cFV^J$ is singular with respect to $C_{JJ}$. Then, there exists a bounded predictable process $\kappa = (\kappa_j; \, j \in J)$ such that $\int_0^\cdot \inner{\kappa(t)}{\ud B(t)}_{\bR^J} = \int_0^\cdot \norm{\ud B(t)}_{\bR^J}$ and $\int_0^\cdot \inner{\kappa (t)}{\ud C_{J j} (t)}_{\bR^J} \equiv 0$ for $j \in J$. It follows that 
\[
K \dfn \int_0^\cdot \inner{\kappa (t)}{\ud P_J (t)}_{\bR^J} = \int_0^\cdot \norm{\ud B(t)}_{\bR^J} \in \Sem(P_J) \subseteq \Sem(P)
\]
is a finite variation process, so $[K, P_i]  = 0$ holds for all $i \in I$; but the mapping $\Sem(P) \ni Z \mapsto [Z, P_i] \in \RKH(C)$ of \eqref{eq:mapping_sem} is one-to-one by assumption, so $K \equiv 0$; this implies  $B \equiv 0$ and establishes the claim.
	
With the above notation, for each $J \in \Fin(I)$  and with $A_J = (A_j; \, j \in J)$,  we have
\[
\int_0^\cdot \norm{\ud A_J(t) }^2_{\ud C_{JJ} (t) } = \int_0^\cdot \sum_{(i, j) \in J \times J} \nu^J_i(t) \ud C_{ij} (t) \nu^J_j(t).
\]

\smallskip
\noindent
\underline{Step 2}: We need to show $A \in \RKH(C)$, and   will argue   this by contradiction; so we  assume that this condition fails, i.e., that there exists  some  $T \in ( 0, \infty)$, with $\bP \bra{\Gamma} > 0$ for  the event
\[
\Gamma \dfn \set{\int_0^T \norm{\ud A(t) }^2_{\ud C (t) }  = \infty }.
\]
Having made this assumption, let us follow through with some of its implications, until such point as a contradiction is reached. A first implication, is that there exist a nondecreasing sequence $(J^n; \, n \in \bN)$ in $\Fin(I)$, and a  sequence $(\Pi^n; \, n \in \bN)$ of predictable \emph{disjoint} sets, such that, with $\eta^n \dfn \nu^{J^n} 1_{\Pi^n}$, the processes 
\begin{align*}
V^n \dfn \int_0^\cdot \sum_{(i, j) \in J^n \times J^n} \eta^n_i(t) \ud C_{ij} (t) \eta^n_j(t) &= \int_0^\cdot \sum_{(i, j) \in J^n \times J^n} \eta^n_i(t) \ud C_{ij} (t) \nu^{J^n}_j(t) \\
&= \int_0^\cdot \inner{\eta^n(t)}{\ud A_{J^n} (t)}
\end{align*}
are finitely-valued for all $n \in \bN$, but also satisfy $\bP \bra{V^n(T) \leq \exp(2^n) \, | \, \Gamma} \leq 2^{-n-1}$, for all $\quad n \in \bN$. With $\Lambda \dfn \bigcap_{n \in \bN} \set{V^n(T) > \exp(2^n)}$, it   then follows  that $\bP \bra{\Lambda \, | \, \Gamma} \geq 1/2$, which in particular   implies that $\bP \bra{\Lambda} > 0$. Next, we define
\[
Y^n \dfn \int_0^\cdot \inner{\eta^n(t)}{\ud P_{J^n} (t)}_{\bR^{J^n}} = V^n + N^n, \quad n \in \bN,
\]
with $V^n \in \cFV$ as above, we have $N^n \dfn \int_0^\cdot \inner{\eta^n(t)}{\ud M_{J^n} (t)}_{\bR^{J^n}} \in \cMloc$,  and     $Y^n \in \Sem(P)$ as well as $[Y^n, Y^n] = V^n = [N^n, N^n]$,   for every $n \in \bN$. Furthermore, since $(\Pi^n; \, n \in \bN)$ is a  sequence of predictable disjoint sets, we have $[Y^n, Y^m] = 0 = [N^n, N^m]$ whenever $m \neq n$. 

We introduce now  the processes
\[
W^n \dfn  \sum_{k=1}^n \frac{1}{2^k}  \int_0^\cdot \frac{\ud Y^k (t)}{1 + V^k(t)} \in \cSem (P), \quad n \in \bN,
\]
and write $W^n = B^n + L^n$, where $B^n \in \cFV$ and $L^n \in \cMloc$,  for each $n \in \bN$. Since $[Y^k, Y^k] = V^k$ for all $k \in \bN$ and $[Y^n, Y^m] = 0$ whenever $m \neq n$, we get 
\[
\bra{W^n, W^n} = \bra{L^n, L^n} = \sum_{k=1}^n  \int_0^\cdot \pare{\frac{1}{2^k} \frac{1}{1 + V^k(t)} }^2 \ud V^k (t) = \sum_{k=1}^n \frac{1}{4^k} \frac{V^k}{1 + V^k}
\]
 for all $n \in \bN$. Furthermore,   the finite variation part of $W^n$ is $B^n = \sum_{k=1}^n 2^{-k} \log \pare{ 1 + V^k }$, for all $n \in \bN$. Since we have $V^n(T) > \exp(2^n)$ on the set $\Lambda$ with $\bP \bra{\Lambda} > 0$, it follows that $B^n(T) > n$ holds on $\Lambda$, for all $n \in \bN$. On the other hand, 
\[
\bra{L^m - L^n, L^m - L^n} = \sum_{k=n+1}^m \frac{1}{4^k} \frac{V^k}{1 + V^k} \leq \sum_{k=n+1}^\infty \frac{1}{4^k} = \frac{1}{2} \frac{1}{4^n}, \quad \forall \, \,n < m,
\]
gives 
\[
\limsup_{k \to \infty} \sup_{m \geq k, \, n \geq k} \bra{L^m - L^n, L^m - L^n} = 0,
\]
which implies that $L \dfn \cSem$-$\lim_{n \to \infty} L^n$ exists and is a continuous local martingale. From Lemma \ref{lem:mapping_well_def_semimart}, we deduce $F \dfn \pare{[L, R_i]; \, i \in I} \in \RKH (C)$. 

We claim that, under condition (1) in the statement of Theorem \ref{thm:struct_rkhs}, this is impossible; to wit, that  \emph{there cannot exist $Z \in \Sem(P)$ with $F = \pare{[Z, P_i]; \, i \in I}$.} Indeed, if such $Z$ existed, then on account of $[L, P_i] = F_i = [Z, P_i]$, which is valid for every $i \in I$, and with $Z^n \dfn \int_0^\cdot 1_{\bigcup_{k=1}^n \Pi^k} (t) \ud Z(t)$,   $n \in \bN$,
we would have 
\[
\bra{Z^n, P_i} = \int_0^\cdot 1_{\bigcup_{k=1}^n \Pi^k} (t) \ud \bra{Z, P_i} (t) = \int_0^\cdot 1_{\bigcup_{k=1}^n \Pi^k} (t) \ud \bra{L, P_i} (t) = \bra{L^n, P_i},
\]
thus also $\bra{Z^n, P_i} = \bra{W^n, P_i}$, for all $n \in \bN$ and $i \in I$. Since   $Z^n \in \Sem(P)$ and $W^n \in \Sem(P)$  hold  for all $n \in \bN$,  and the mapping \eqref{eq:mapping_sem} is one-to-one, the identity  $Z^n = W^n = B^n + L^n$ would then follow for all $n \in \bN$. But  $B^n(T) > n$ holds on the set $\Lambda$ of positive probability $\bP \bra{\Lambda} > 0$, for all $n \in \bN$, so it is impossible for  the sequence $(Z^n; \, n \in \bN)$  to converge  in $\cSem$; however, by its definition this sequence should converge to $Z$. 

We have reached the desired  contradiction. This shows that the condition $A \in \RKH(C)$ is necessary for the mapping \eqref{eq:mapping_sem} to be a bijection.

\smallskip
\noindent
\underline{Step 3}: Finally, we observe that  under the condition (1), which is equivalent to the structural condition (2) as   just argued, the topology on $\RKH(C)$ is coarser than the topology on $\Sem(P)$. On the other hand, and as we have seen, Proposition \ref{prop:semimart_closure_charact} shows that the inverse of the mapping  \eqref{eq:mapping_sem} is given by $\RKH(C) \ni F \longmapsto \int_0^\cdot \inner{\ud F (t)}{\ud A (t)}_{\ud C (t)} + \int_0^\cdot \inner{\ud F (t)}{\ud M (t)}_{\ud C (t)}  \in \Sem(P)$. Since
\[
\int_0^\cdot \abs{ \inner{\ud F (t)}{\ud A (t)}_{\ud C (t)} } \leq \sqrt{\int_0^\cdot \norm{\ud A(t)}^2_{\ud C(t)}} \sqrt{\int_0^\cdot \norm{\ud F(t)}^2_{\ud C(t)}},
\]
the topology on $\Sem(P)$ is coarser than the topology on $\RKH(C)$. It follows that the spaces $\Sem(P)$ and  $\RKH(C)$ are topologically isomorphic.
\end{proof}

\begin{remark}
Even though it is contained in the proof of Theorem \ref{thm:struct_rkhs} above, we provide here an easy example demonstrating that, when the structural condition $A \in \RKH(C)$ fails, the mapping $\Sem (P) \ni Z \mapsto ([Z, P_i]; \, i \in I) \in \RKH (C)$ may fail to be both one-to-one and onto. Consider the one-dimensional semimartingale $P = A + M$, where $M$ is a standard Brownian motion and $A(t) = 3 t^{1/3} + B(t)$, where $B$ is a deterministic nondecreasing process with $0 = B(0) < B(1)$, singular with respect to Lebesgue measure $\Leb$. Then, $C(t) = t$ for $t \in \bR_+$, and $\RKH(C)$ consists of all $F \equiv \int_0^\cdot f(t) \ud t$ with $\int_0^T |f(t)|^2 \ud t < \infty$, for all $T > 0$. Pick a deterministic $\{ 0, 1\}$-valued process $b$ such that $\Leb [b \neq 0] = 0$ and $\int_0^\cdot b(t) \ud B(t) = B$. Then, $Z \equiv \int_0^\cdot \beta(t) \ud P(t) = B \in \Sem(P)$, but $[Z, P] \equiv 0$. This implies that $\Sem (P) \ni Z \mapsto [Z, P] \in \RKH (C)$ is not one-to-one, since also $[0, P] = 0$, and $B \neq 0$. Furthermore, let $f: \bR_+ \mapsto \bR$ be defined by $f(t) = t^{- 1 /3} \one_{(0, \infty)} (t)$ for $t \in \bR$, and note that $\int_0^\cdot |f(t)|^2 \ud t = \int_0^\cdot t^{-2/3} \ud t$ is finitely-valued, so that $F \equiv \int_0^\cdot f(t) \ud t \in \RKH(C)$. If $Z \in \Sem (P)$ with $[Z, P] = F$ existed, it would have to be $Z = \int_0^\cdot f(t) \ud P(t)$. While $\int_0^\cdot f(t) \ud M(t)$ is well-defined, the putative finite-variation part satisfies $\int_0^T f(t) \ud A(t) \geq \int_0^T t^{-1} \ud t = \infty$ for all $T > 0$, which implies that the mapping $\Sem (P) \ni Z \mapsto [Z, P] \in \RKH (C)$ is not onto.
\end{remark}

\section{Applications to Mathematical Finance}
\label{sec:math_fin}

\subsection{Simple trading and market viability}

For the purposes of Section \ref{sec:math_fin}, $P \equiv (P_i; \, i \in I)$ will be denoting a collection of continuous stochastic processes, with each $P_i$ modelling the price movement of security $i \in I$ in the market, appropriately discounted by a strictly positive num\'{e}raire.

Define $\cXs$ as the class of all \emph{nonnegative} wealth processes of the form
\begin{equation} \label{eq:wealth_simple}
x + \int_0^\cdot \sum_{j \in J} \theta_j(t) \ud P_j (t),
\end{equation}
where $x \in \bR_+$, $J \in \Fin (I)$ and $\theta_J \equiv (\theta_j; j \in J)$ in \eqref{eq:wealth_simple} predictable and \emph{simple}, i.e., consists of a finite number of piecewise constant (in time) parts. As the integrands involved are simple, the stochastic integrals may be defined in the usual pathwise sense. 


Recall that $\FV_\succeq$ denotes the class of all nondecreasing right-continuous processes $K$ with $K(0) = 0$. With the above understanding, define
\begin{equation} \label{eq:heding_value_simple}
\xs (K) \dfn \inf \set{x > 0 \such \exists X \in \cXs \text{ with } X(0) = x \text{ and } X \geq K}, \quad K \in \FV_\succeq.
\end{equation}
In words, $\xs (K)$ is the hedging value of the stream $K \in \FV_\succeq$ upon use of simple trading.

\begin{definition} \label{defn:viable}
We say that the market is \textbf{viable} if
\[
K \in \cK, \quad \xs(K) = 0 \quad \Longrightarrow \quad K \equiv 0.
\]
\end{definition}

Viability states that it is not possible to finance a non-trivial stream $K$, using simple predictable admissible strategies that invest in a finite number of assets,  starting with positive initial capital arbitrarily near zero. It can be shown as in\footnote{Condition NA$_1$ in \cite{MR2732838} involves only ``European contingent claim'' streams of the form $K = g 1_{[T, \infty)}$ for $T > 0$ and $\cF(T)$-measurable $g \geq 0$, but it is straightforward to see that the definitions are equivalent.} \cite[Proposition]{MR2732838} that market viability is equivalent to the requirement that
\begin{equation}
\label{eq:bdd_prob}
\lim_{\ell \to \infty} \sup_{X \in \cXs, \, X (0) = 1} \bP \bra{ X (T) > \ell } = 0, \quad \forall \, T \in \bR_+.
\end{equation}
i.e., that $\set{X (T) \such X \in \cXs \text{ with } X (0) = 1}$ is bounded in $\bP$-measure for all $T \in \bR_+$. 

In particular, \eqref{eq:bdd_prob} implies that the market consisting of assets $(P_i; \, i \in I)$ is viable if, and only if, all markets consisting of assets $(P_i; \, i \in Q)$ are viable, for all $Q \in \Cou(I)$.

\begin{definition} \label{defn:loc_mart_defl}
A process $Y$ will be called a \textbf{local martingale deflator} if $Y > 0$, $Y(0) = 1$ and all processes $Y$ and $Y P \equiv \pare{Y P_i; \, i \in I}$ are local martingales. The class of all such local martingale deflators will be denoted by $\cY$.
\end{definition}

Suppose that $\cY \neq \emptyset$. First, note that every $P_i$ for $i \in I$ is a semimartingale. This follows from It\^{o}'s formula and the product rule, using the fact that we can write $P_i = (1/Y) Y P_i$, and both processes $1/Y$ and $Y P_i$ are semimartingales. Now, pick $Y \in \cY$ and $K \in \FV_\succeq$ with $\xs(K) < \infty$. For $x > \xs(K)$, pick $X \in \cXs$ with $X(0) = x$ and $X \geq K$. Since $X$ is a stochastic integral with respect to a finite number of semimartingale integrators from $P$ Using integration-by-parts, it is straightforward to show that $Y X$ is a local martingale. Now, $\int_0^\cdot Y(t) \ud K(t) - Y K = \int_0^\cdot K(t-) \ud Y(t)$ is also a local martingale, which gives that $Z \dfn Y(X- K) + \int_0^\cdot Y(t) \ud K(t)$ is a nonnegative local martingale, thus a supermartingale, and the supermartingale convergence theorem provides a limit at infinity $Z(\infty) \geq \int_0^\infty Y(t) \ud K(t)$. Then, the optional sampling theorem gives
\[
\expec \bra{\int_0^\infty Y(t) \ud K(t)} \leq \expec [Z(\infty)] \leq Z(0) = Y(0) X(0) = x.
\]
Taking supremum over all $Y \in \cY$ and infimum over all $x > \xs(K)$ in the left- and right-hand sides of the above inequality, respectively, we obtain
\begin{equation} \label{eq:hedging_simple_easy}
\sup_{Y \in \cY} \expec \bra{\int_0^\infty Y(t) \ud K(t)} \leq \xs (K), \quad K \in \FV_\succeq. 
\end{equation}
	
Using a combination of \cite[\S 2.3]{MR2832419} and \cite[Theorem 4]{MR2732838}, the following follows in the case where $I$ has finite cardinality; we show that it is also true for arbitrary index sets $I$. The appellation of structural condition for \eqref{eq:struct_fin} goes at least as back as \cite{MR1353193}.

\begin{theorem}[Fundamental theorem] \label{thm:ftap}
The following statements are equivalent:
\begin{enumerate}
	\item The market is viable in the sense of Definition \ref{defn:viable}.
	\item There exists a local martingale deflator as in Definition \ref{defn:loc_mart_defl}: $\cY \neq \emptyset$.
	\item The component processes of $P \equiv (P_i; \, i \in I)$ are semimartingales. Furthermore, with Doob-Meyer decompositions  $P_i = A_i + M_i$, with $A_i \in \cFV$ and $M_i \in \cMloc$ for $i \in I$, the structural condition $A \equiv (A_i; \, i \in I) \in \RKH (C)$ holds:
\begin{equation} 
\label{eq:struct_fin}
\int_{0}^T \norm{\ud A(t)}^2_{\ud C (t)} < \infty, \quad \forall \ T \in \bR_+.
\end{equation}
\end{enumerate}
\end{theorem}

\begin{proof}
We start with the implication $(3) \Rightarrow (2)$. Under the assumptions of statement (3), and in view of Theorem \ref{thm:struct_rkhs} and of \eqref{eq:FV2C}--\eqref{eq:stoch_rkhs_approx}, there exists  $Z \in \Sem(P)$ such that $[Z, P_i] = A_i$ holds for all $i \in I$. Write  $Z = B + L$ for appropriate $B \in \cFV$ and $L \in \cMloc$, and introduce the  strictly positive local martingale\footnote{Here and until the end of Section \ref{sec:math_fin}, we use ``$\cE (\cdot)$'' to denote stochastic exponential.} $Y = \cE(- L)$ with $Y(0) = 1$. Since
\[
\cE(-L) \cE(P_i) = \cE(-L + P_i - [L, P_i]) = \cE(-L + M_i), \quad i \in I,
\]
as follows from the Yor formula \cite[Exercise IV.(3.11)]{MR1725357} and the fact that $[L, P_i] = [Z, P_i] = A_i$ holds  for all $i \in I$, we deduce that $Y S_i = S_i(0) \cE(-L) \cE(R_i)$ is a local martingale, again for every  $i \in I$, showing that $Y \in \cY$.

For the implication  $(2) \Rightarrow (1)$, assume (2), let $Y \in \cY$, and pick $K \in \FV_\succeq$ such that $\xs(K) = 0$. Then, $\expecp \bra{\int_0^\infty Y(t)  \ud K(t)} = 0$ follows from \eqref{eq:hedging_simple_easy}, implying that, $\bP$-a.e., $\int_0^\infty Y(t) \ud K(t) = 0$.  Since $Y$ is strictly positive and $K \in \FV_\succeq$, it holds that $K \equiv 0$. Market viability follows.

Finally, we broach    the implication $(1) \Rightarrow (3)$. The viability of the entire market  implies, in particular, that every  sub-market with a finite number of assets is viable in the sense of Definition \ref{defn:viable}. In view of \cite[\S 2.3]{MR2832419}, each $P_i$, $i \in I$, is a semimartingale. Then, using the fact that stochastic integrals of continuous semimartingales can be approximated in probability uniformly on compact time intervals, implies that condition \eqref{eq:bdd_prob} also holds when $\cXs$ is replaced by the family $\cX$ consisting of all nonnegative stochastic integrals using a finite number of integrators from $P$. Therefore,  [Theorem 4]\cite{MR2732838} allows us to deduce  that $A_J \in \RKH(C_{JJ})$ holds for every $J \in \Fin(I)$. Suppose now, by way of contradiction, that $A \notin \RKH(C)$. In this case, and in view of Remark \ref{rem:increasing_bddness_fail}, there exist a real number $T > 0$ and an increasing sequence $(J^n; \, n \in \bN)$ in $\Fin(I)$, such that $\bP \bra{\lim_{n \to \infty} G^n(T) = \infty} > 0$ holds with
\[
G^n \dfn \frac{1}{2} \int_{0}^\cdot \norm{\ud A_{J^n} (t)}^2_{\ud C_{J^n J^n} (t)}, \quad n \in \bN.
\] 
For each $n \in \bN$, Theorem \ref{thm:struct_rkhs} gives us again a process $Z^n \in \Sem(P_{J^n})$ with     $[Z^n, R_j] = A_j$ for all $j \in J^n$, and we note that $Z^n = 2 G^n + L^n$ holds for $L^n = \int_0^\cdot \inner{\ud A_{J^n} (t)}{\ud M_{J^n} (t)}_{\ud C_{J^n J^n} (t)} \in \cMloc$ with $[L^n, L^n] = 2 G^n$.  Then for the strictly positive continuous semimartingale
\[
W^n \dfn \cE (Z^n) \in \Sem(P_{J^n}),
\]	
there exists a process $X^n \in \cXs$ with the property $\bP \bra{X^n(T) \leq W^n(T) - 1} \leq 1/n$.

We claim that the resulting sequence $(X^n(T) ; \, n \in \bN)$ fails to be bounded in $\bP$-measure, contradicting \eqref{eq:bdd_prob} and, therefore, the fact that the market is viable.  To prove this claim, it is enough to show that $(W^n(T) ; \, n \in \bN)$ fails to be bounded in $\bP$-measure. Indeed, we note
\[
\log W^n = \log \cE (Z^n) = Z^n - \frac{1}{2} [Z^n, Z^n] = G^n + L^n, \quad n \in \bN 
\]
and recall that $[L^n, L^n] = 2 G^n$ holds for all $n \in \bN$. The Dambis-Dubins-Schwarz representation (e.g., \cite[Theorem 3.4.6 and Problem 3.4.7]{MR1121940}), combined with the scaling property of Brownian motion, imply that for every $n \in \bN$ there exists a Brownian motion $\beta^n$, on a possibly enlarged filtered probability space, such that $\log W^n = G^n  + \sqrt{2} \beta^n (G^n)$. The strong law of large numbers for Brownian motion gives
\[
\lim_{n \to \infty} \bP \bra{ \frac{\beta^n (G^n (T)) }{ G^n(T) } \leq - \frac{1}{2 \sqrt{2}}, \quad \lim_{m \to \infty} G^m(T) = \infty} = 0.
\]
Since $\log W^n = G^n  + \sqrt{2} \beta^n (G^n)$ holds for all $n \in \bN$, we obtain  
\[
\lim_{n \to \infty} \bP \bra{ \frac{\log W^n (T)}{ G^n(T) }  \leq \frac{1}{2}, \quad \lim_{m \to \infty} G^m(T) = \infty } = 0,
\]
in turn implying that
\[
\lim_{n \to \infty} \bP \bra{\log W^n(T) > \ell \ \big| \  \lim_{m \to \infty} G^m(T) = \infty} = 1
\]
holds for all $\ell \in \bN$, This  shows that $(W^n(T); \, n \in \bN)$ fails to be bounded in $\bP$-measure, and completes the argument.
\end{proof}

\subsection{Structure of local martingale deflators}

In the proof of implication $(3) \Rightarrow (2)$ in Theorem \ref{thm:ftap}, a specific local martingale deflator was constructed. To recapitulate and set some notation to be used below, with $M^A \in \Sem(M) \subseteq \cMloc$ such that $[M^A, M_i] = A_i$, for all $i \in I$, we have $\cE(- M^A) \in \cY$. The next result gives the structure of all local martingale deflators in a viable market, and is well-known in the case where $I$ has finite cardinality; see for example, \cite[Theorem 1]{MR1353193}, for the corresponding decomposition of densities of equivalent local martingale measures (the generalisation to local martingale deflators for finite-asset markets is straightforward).

\begin{proposition} \label{prop:mulp_decomp_local_mart_defl}
Suppose that the market is viable, and let $M^A = \int_0^\cdot \inner{\ud A(t)}{\ud M(t)}_{\ud C(t)}$ be the unique continuous local martingale in $\Sem(M)$ with $[M^A, M_i] = A_i$, $i \in I$. Then, the collection $\cY$ of local martingale deflators contains   exactly those  processes $Y$ of the form
\[
Y = \cE(- M^A) \cE (L) = \cE(-M^A + L),
\]
where $L \in \Mloc$ is such that $L(0) = 0$, $\Delta L > -1$, and $[L, M_i] = 0$ for all $i \in I$.
\end{proposition}

\begin{proof}
We cast an arbitrary strictly positive local martingale $Y$ in the form  $Y = \cE(M^F + L) = \cE(M^F) \cE(L) ,$  where $F \in \RKH (C)$ and where $L$ is a local martingale with  $L(0) = 0$, $\Delta L > -1$, and $[L, M_i] = 0$  for all $i \in I$. From the Yor formula,  we obtain
\[
Y S_i = \cE(M^F + L) S_i(0)  \cE(R_i) = S_i(0)  \cE (M^F + L + M_i + A_i + F_i), \quad i \in I,
\]
since $[L, R_i] \equiv 0$ and $[M^F, R_i] = F_i$, thus also 
\[
M^F + L + R_i + [M^F + L, R_i] = M^F + L + R_i + F_i = M^F + L + M_i + A_i + F_i. 
\]
We deduce that $Y S^i$ is a local martingale for all $i \in I$ if, and only if, $A_i + F_i = 0$ holds for all $i \in I$, i.e.,   $F = -A$, concluding the argument.
\end{proof}

\subsection{General wealth-consumption processes}
\label{subsec:gen_wealth_cons}

Assume that the market is viable. Given the semimartingale property of $P$ in Theorem \ref{thm:ftap}, one may define wealth processes using general stochastic integrals. Define $\cX$ as the set of all \emph{nonnegative} processes $x + Z$, where $x \in \bR_+$ and $Z \in \Sem(P)$. Since market viability is equivalent to the structural condition $A \in \RKH(C)$, Proposition \ref{prop:semimart_closure_charact} implies that $\cX$ coincides with all nonnegative processes of the form $x + \int_0^\cdot \inner{\ud F(t)}{\ud P(t)}_{\ud C(t)} = x + \int_0^\cdot \inner{\ud F(t)}{\ud A(t)}_{\ud C(t)}  + \int_0^\cdot \inner{\ud F(t)}{\ud P(t)}_{\ud M(t)} $, where $x \in \bR_+$ and $F \in \RKH(C)$. In fact, for future reference let us define
\begin{equation} \label{eq:wealth_cons_extended}
X^{x, F, K} \dfn x + \int_0^\cdot \inner{\ud F(t)}{\ud P(t)}_{\ud C(t)} - K, \qquad x \in \bR, \ F \in \RKH(C), \ K \in \FV_\succeq. 
\end{equation} 
with the interpretation of a wealth process where $K$ is an aggregate capital withdrawal stream.
We simply write $X^{x, F}$ for $X^{x, F, 0}$, i.e., when no capital withdrawal process is present. With this notation, $\cX = \set{X^{x, F} \geq 0 \  | \  x \in \bR_+, \, F \in \RKH(C)}$.

In the present financial setting, a remark on the interpretation of the integrands in $\RKH (C)$ is in order. While the predictable process $\theta$ in \eqref{eq:wealth_simple} denotes positions held in each of the assets, the components of an integrand $F = (F_i; \, i \in I) \in \RKH(C)$ in \eqref{eq:wealth_cons_extended} carry the interpretation of aggregate covariations of the resulting wealth process with the individual assets. One may argue that, as an input, covariations are as natural as (or even more appropriate than) positions: one typically cares about the sensitivity of investment with respect to asset price movements, and this is exactly what integrands in $\RKH (C)$ encode. 

\subsection{Optional decomposition}

Having a flavour of a ``uniform'' (over local martingale measures, or local martingale deflators, as here) Doob-Meyer decomposition , the optional decomposition theorem has been vital in the development of Mathematical Finance, especially in the context of the hedging duality, taken up later on in \S \ref{subsec:hedging}. The earliest contribution dealing with a finite numbers of It\^{o}-process  integrators is \cite{MR1311659}, later generalised in \cite{MR1402653}, \cite{MR1469917}, for general semimartingale integrators. The paper \cite{MR1651229} deals with local martingale deflators, instead of local martingale measures.

We shall present an infinite-asset generalisation in Theorem \ref{thm:odt}. We begin with a relatively simple observation.

\begin{lemma} \label{lem:numeraire_discount}
Assume that the market is viable. For any $x \in \bR$ and $F \in \RKH (C)$, we have $\cE(- M^A) X^{x, F} = x + M^H$ for some $H \in \RKH (C)$. 
\end{lemma}

\begin{proof}
Since $\cE(- M^A) - 1 = - \int_0^\cdot \cE(- M^A) (t) \ud M^A(t) \in \Sem(M)$ and $\Sem(M)$ is isometric to $\RKH(C)$, it suffices to show that $\cE(- M^A) X^{0, F} \in \Sem(M)$. Since $[M^A, X^{0, F}] = [M^A, M^F] = \int_0^\cdot \inner{\ud A(t)}{\ud F(t)}_{\ud C(t)}$, we obtain
\[
\bra{\cE(- M^A), X^{0, F}} = - \int_0^\cdot \cE(- M^A)(t) \ud \bra{M^A, M^F} (t),
\]
and integration-by-parts gives
\begin{align*}
\cE(- M^A) X^{0, F} = - \int_0^\cdot X^{0, F} (t) \cE(- M^A) (t) \ud M^A(t) + \int_0^\cdot \cE(- M^A)(t) \ud M^F(t), 
\end{align*}
which shows that, indeed $\cE(- M^A) X^{0, F} \in \Sem(M)$.
\end{proof}

Let $Y \in \cY$, and write $Y = \cE(- M^A) \cE(L)$ as in Proposition \ref{prop:mulp_decomp_local_mart_defl}. Then, according to Lemma \ref{lem:numeraire_discount}, for any $x \in \bR$ and $F \in \RKH (C)$ we have $Y X^{x, F} = \cE(L) (x + M^H)$ for some $H \in \RKH(C)$, and since $[L, M^H] = 0$ we obtain that $Y X^{x, F} \in \Mloc$. Since the process
\[
Y K = \int_0^\cdot K(t-) \ud Y(t) + \int_0^\cdot Y(t) \ud K(t)
\]
is clearly a local submartingale for all $K \in \FV_\succeq$, we further obtain that $Y X^{x, F, K}$ is a local supermartingale for all $x \in \bR$, $F \in \RKH (C)$ and $K \in \FV_\succeq$. Furthermore, given $Y \in \cY$, $x \in \bR$, $F \in \RKH (C)$ and $K \in \FV_\succeq$, $Y X^{x, F, K}$ is a local martingale if and only if $K \equiv 0$.

\begin{theorem}[Optional decomposition] \label{thm:odt}
Suppose that $\cY \neq \emptyset$, and let $X$ be a nonnegative stochastic process with $X(0) = x \in \bR_+$. Then, the following statements are equivalent:
\begin{enumerate}
	\item $Y X$ is a local supermartingale for all $Y \in \cY$.
	\item It holds that $X = X^{x, F, K}$, for some $F \in \RKH(C)$ and $K \in \FV_\succeq$. 
\end{enumerate}
Under any of the equivalent conditions above, $F = ([X, P_i]; i \in I) \in \RKH (C)$.
\end{theorem}

\begin{proof}
Implication $(2) \Rightarrow (1)$ was already discussed just before the statement of Theorem \ref{thm:odt}. The proof of implication $(1) \Rightarrow (2)$ mostly follows the development of \cite{MR3384117}, but we provide a somewhat different argument, briefly explaining the steps.

Assume condition (1), and set $F = ([X, P_i]; i \in I) \in \cFV^I$. By Proposition \ref{prop:semimart_closure_charact}, $F \in \RKH (C)$; then, since $\cY \neq \emptyset$ implies to $A \in \RKH(C)$ by Theorem \ref{thm:ftap}, we obtain from Theorem \ref{thm:struct_rkhs} the existence of a unique $Z \in \Sem(P)$ with $[Z, P_i] = F_i$, for all $i \in I$. Define the locally bounded from above process $K \dfn Z + x - X$. It then follows that $[K, M_i] = [K, P_i] = [Z, P_i] - [X, P_i] = F_i - F_i = 0$ holds for all $i \in I$. Furthermore, statement (1) and the discussion after Lemma \ref{lem:numeraire_discount} imply that $Y K$ is a local submartingale for all $Y \in \cY$. We need to show that $K \in \FV_{\succeq}$. Equivalently, and upon defining
\begin{equation} \label{eq:B_from_K_odt_proof}
B \dfn \int_0^\cdot \cE(- M^A)(t) \ud K(t) = \cE(- M^A) K + \int_0^\cdot K(t-) \cE(- M^A)(t -) \ud M^A(t),
\end{equation}
where the fact that $[K, M^A] = 0$ was used to obtain the right-hand-side equality, we need to show that $B  \in \FV_{\succeq}$.

Since $K$ is a local submartingale with $[K, M_i] = 0$ for all $i \in I$, it follows that $B$ is a local submartingale with $[B, M_i] = 0$ for all $i \in I$. In particular, if $N \in \cMloc$ denotes the uniquely-defined continuous local martingale part of $B$, we have $[N, M_i] = 0$ for all $i \in I$, and $B - N$ is local submartingale which is a purely discontinuous, in the sense that $[L, B - N] = 0$ holds for all $L \in \cMloc$. Since $Y K$ is a local submartingale for all $Y \in \cY$, integration-by-parts and \eqref{eq:B_from_K_odt_proof}, using that $\cE( - m N) \cE (- M^A) \in \cY$ and that $[N, M_i] = 0$ for all $i \in I$, gives that the process $\cE(- m N) B$ is a local submartingale for all $m \in \bN$. Again, using integration-by-parts and the fact that $[B, N] = [N, N]$, we obtain
\[
\cE(-m N) B = - m \int_0^\cdot \cE ( - m N) (t-) B(t-) \ud N(t) + \int_0^\cdot \cE (- m N) (t-) \ud \pare{B(t) - m [N, N](t)}
\]
Since above process is local submartingale for all $m \in \bN$, it follows that $B - m [N, N]$ is local supermartingale for all $m \in \bN$, which is only possible if $[N, N] = 0$, i.e., if $N = 0$. Therefore, it follows that $[L, B] = 0$ for all $L \in \cMloc$, and  $L B$ is a local submartingale for all purely discontinuous local martingales $L$ with $\Delta L > -1$. Directly applying \cite[Lemma 2.1]{MR3384117}, we obtain that $B$ has to actually be nondecreasing, i.e., $B \in \FV_{\succeq}$, which completes the argument.
\end{proof}

\begin{remark} \label{rem:odt_loc_mart}
Suppose that $\cY \neq \emptyset$, and let $X$ be a stochastic process with $X(0) = x \in \bR$ that is locally bounded from below. Then, the following statements are equivalent:
\begin{enumerate}
\item $Y X$ is a local martingale for all $Y \in \cY$.
\item It holds that $X = X^{x, F}$, for some $F \in \RKH(C)$.
\end{enumerate}
Indeed, implication $(2) \Rightarrow (1)$ was discussed before the statement of Theorem \ref{thm:odt}. Assuming condition (1), and since local martingales that are locally integrable from below are local supermartingales, Theorem \ref{thm:odt} gives that $X = X^{x, F, K}$ holds for some $F \in \RKH(C)$ and $K \in \FV_\succeq$. Again, condition (1) and the discussion before the statement of Theorem \ref{thm:odt} implies that $K = 0$, and condition (2) follows.
\end{remark}



\subsection{Hedging}
\label{subsec:hedging}

We assume that the market is viable, which allows use of wealth-consumption processes as explained in \S \ref{subsec:gen_wealth_cons}.

A wealth-consumption process $X \equiv X^{x, F, G}$, for $x \in  \bR_+$, $F \in \RKH(C)$ and $G \in \FV_\succeq$ is said to \textbf{hedge} for a given $K \in \FV_\succeq$ if $X \geq K$ holds; furthermore, such $X$ will be called a \textbf{minimal hedge for $K$} if $X \leq Z$ holds whenever $Z$ is any other hedge for $K$. If $X^{x, F, G}$, for $x \in  \bR_+$, $F \in \RKH(C)$ and $G \in \FV_\succeq$ is a hedge for $K \in \FV_\succeq$, then the ``pure wealth'' process $X^{x, F}$ is also certainly a hedge; however, the minimal hedge for $K$ may also involve capital withdrawal.

With the above understanding, and in accordance with \eqref{eq:heding_value_simple}, define the \textbf{hedging value}
\begin{equation} \label{eq:hedging_value}
x (K) \dfn \inf \set{x > 0 \such \exists F \in \RKH(C) \text{ with } X^{x, F} \geq K}, \quad K \in \FV_\succeq.
\end{equation}
Since $\cXs \subseteq \cX$, $x(K) \leq \xs (K)$ holds for all $K \in \FV_\succeq$. As in the proof of \eqref{eq:hedging_simple_easy},
\begin{equation} \label{eq:hedging_easy}
\sup_{Y \in \cY} \expec \bra{\int_0^\infty Y(t) \ud K(t)} \leq x (K), \quad K \in \FV_\succeq 
\end{equation}
holds. In particular, $x(K) = 0$ for $K \in \FV_\succeq$ implies $K \equiv 0$, which is a seemingly stronger condition to market viability. It follows that $x(K) = 0$ if and only if $\xs(K) = 0$; however, it is straightforward to construct examples of $K \in \FV_\succeq$ where the strict inequality $x(K) < \xs(K)$ is valid.

The next auxiliary result can be seen as a ``dynamic'' version of \eqref{eq:hedging_value}.

\begin{lemma}
Let $X$ be any hedge of $K \in \FV_\succeq$. Then, it holds that
\begin{equation} \label{eq:hedging_easy_dynamic}
K(s) + \esssup_{Y \in \cY} \expec \bra{ \int_{(s, \infty)} \frac{Y(t)}{Y(s)} \ud K(t) \ \Big| \ \cF(s)} \leq X(s), \quad s \in \bR_+.
\end{equation}
\end{lemma}

\begin{proof}
Fix $s \in \bR_+$. For any $Y \in \cY$ and any stopping time $T \geq s$,
\[
Y(T) X(T) \geq Y(T) K(T) = Y(s) K(s) + \int_s^T Y(t) \ud K(t) + \int_s^T K(t-) \ud Y(t).
\]
A standard localisation argument applied to the local martingale $\int_0^\cdot K(t-) \ud Y(t)$ combined with the monotone convergence theorem, using also the supermartingale property of $Y X$, gives
\[
Y(s) X(s) \geq Y(s) K(s) + \expec \bra{\int_s^\infty Y(t) \ud K(t) \ \Big| \ \cF(s)}.
\]
Upon dividing with $Y(s)$ throughout in the last equality, and then taking essential supremum over all $Y \in \cY$, \eqref{eq:hedging_easy_dynamic} follows.
\end{proof}

The next result implies in particular that the inequality in \eqref{eq:hedging_easy} is an actual equality.

\begin{theorem}[Hedging duality] \label{thm:hedging}
Assume that the market is viable in the sense of Definition \ref{defn:viable}. Then, it holds that
\begin{equation} \label{eq:hedging_full}
x(K) = \sup_{Y \in \cY} \expec \bra{ \int_0^\infty Y(t) \ud K(t)}, \quad K \in \FV_\succeq.
\end{equation}
Furthermore, if $x(K) < \infty$ for $K \in \FV_\succeq$, then there exists a minimal $X$ hedge for $K$ with $X(0) = x(K)$; for this minimal hedge $X$ it holds that
\begin{equation} \label{eq:minimal_hedge}
X(s) = K(s) + \esssup_{Y \in \cY} \expec \bra{ \int_{(s, \infty)} \frac{Y(t)}{Y(s)} \ud K(t) \ \Big| \ \cF(s)}, \quad s \in \bR_+.
\end{equation}
\end{theorem}

\begin{proof}
Given the validity of the optional decomposition theorem \ref{thm:odt}, the proof of this result is standard, and we only sketch it. Let $z(K)$ be the quantity on the right-hand-side of \eqref{eq:hedging_full}, so that $z(K) \leq x(K)$ holds for $K \in \FV_\succeq$ by \eqref{eq:hedging_easy}. If $z(K) = \infty$, \eqref{eq:hedging_full} is trivially satisfied. Assume then that $z(K) < \infty$, and define $Z(s)$ for $s \in \bR_+$ as the right-hand-side of \eqref{eq:minimal_hedge}. Note that $Z(0) = z(K)$. One may show, for example following the arguments in \cite[Proposition 4.3]{MR1402653} (replacing local martingale densities there with local martingale deflators) that, for every $Y \in \cY$, $Y Z$ is a supermartingale  with right-continuous expectation; in particular, since $\cY \neq \emptyset$, $Z$ admits a right-continuous modification, which we still denote by $Z$. Then, according to the optional decomposition theorem \ref{thm:odt}, there exist $F \in \RKH(C)$ and $G \in \FV_\succeq$ such that $X^{z(K), F, G} = Z \geq K$. In particular, $x(K) \leq z(K)$, which implies that $x(K) = z(K)$. More generally, $Z$ is a hedge for $K$ and, since any hedge of $K$ has to be larger than $Z$ by \eqref{eq:hedging_easy_dynamic}, we obtain that $Z$ is a minimal hedge, and \eqref{eq:minimal_hedge} follows.
\end{proof}

\subsection{Completeness}
We are assuming throughout that the market is viable. We interpret a pair $(T, g)$, where $T \in \bR_+$ and $g \in \lzp (\cF(T))$ (i.e., $g$ is a nonnegative $\cF(T)$-measurable random variable as a \emph{European contingent claim}, where the \emph{payoff} $g$ is to be collected at \emph{maturity} $T$. Any such pair $(T, g)$ may be identified with the liability stream $g 1_{[T, \infty)}$, which makes $x (g 1_{[T, \infty)})$ its \emph{hedging capital}. More precisely, as a corollary of Theorem \ref{thm:hedging}, we have
\[
x (g 1_{[T, \infty)}) = \sup_{Y \in \cY} \expec \bra{Y(T) g};
\]
furthermore, if $x \equiv x (g 1_{[T, \infty)}) < \infty$, there exists a minimal hedge, i.e., a wealth-consumption process $X \equiv X^{x, F, G}$ such that $X(T) \geq g$, which is also minimal in having this property.

\smallskip

As part of the next definition, a given wealth process $X \in \cX$ will be called \textbf{maximal at $T \in \bR_+$} if, whenever $Z \in \cX$ is such that $Z(0) = X(0)$ and $X(T) \leq Z(T)$, it actually holds that $X(T) = Z(T)$.

\begin{definition} \label{defn:complete}
A viable market will be called \textbf{complete} if, whenever $T \in \bR_+$ and $g \in \lzp (\cF(T))$ are such that $x \pare{g 1_{[T, \infty)}} < \infty$, then there exists $X \in \cX$ that is maximal at $T$ and such that $X(T) = g$.
\end{definition}

The maximality in Definition \ref{defn:complete} is there to avoid use of suicide strategies for replication of contingent claims. It is the equivalent of asking that wealth processes replicating bounded contingent claims should be bounded, that appears in ``classical'' definitions of completeness.

The next result, important enough to usually go by the appellation ``second fundamental theorem'', goes at least as back as \cite{MR622165}.

\begin{theorem} \label{thm:ftap2}
Assume that the market is viable. Then, the market is complete if, and only if, there exists a unique local martingale deflator.
\end{theorem}

\begin{proof}
Assume first that there exists a unique local martingale deflator: $\cY = \set{Y}$. Let $T \in \bR_+$ and $g \in \lzp (\cF(T))$ be such that $x \equiv x (g 1_{[T, \infty)}) < \infty$. By Theorem \ref{thm:hedging}, $x = \expec \bra{ Y(T) g}$. Define a nonnegative martingale $N$ via $N(t) = \expec \bra{ Y(T) g \such  \cF(t)}$, for all $t \in [0, T]$. Since $\cY = \set{Y}$, it follows from Theorem \ref{thm:hedging} that the minimal hedge $X$ associated with $g 1_{[T, \infty)}$ satisfies $Y X = N$. Since $\cY = \set{Y}$ and $N$ is a (local) martingale, Remark \ref{rem:odt_loc_mart} implies that $X \in \cX$. We claim also that $X$ is maximal: indeed, if $Z \in \cX$ satisfies $Z(0) = x = X(0)$ and $X(T) \leq Z(T)$, then $\expec \bra{Y(T) Z(T)} \leq Y(0) Z(0) = x = \expec \bra{Y(T) X(T)}$, which combined with $Y(T) > 0$ gives $X(T) = Z(T)$. Since $T \in \bR_+$ and $g \in \lzp (\cF(T))$ with $x (g 1_{[T, \infty)}) < \infty$ are arbitrary, market completeness follows.

Assume now that the market is (viable and) complete. By way of contradiction, suppose that there is more than one local martingale deflators. In view of Proposition \ref{prop:mulp_decomp_local_mart_defl}, there exists $T > 0$ and $L \in \Mloc$ with $\bP \bra{L(T) = 0} < 1$ and $\bra{L, M_i} = 0$, for all $i \in I$. It is a straightforward to check that we may additionally assume that $L$ satisfies $|L| \leq 1/2$. Define $g \in \lzp (\cF(T))$ via $g \dfn \pare{1/2 + L(T)} / \cE(-M^A) (T)$. Note that
\begin{equation} \label{eq:complete_help}
\expec \bra{\frac{Y(T)}{Y(s)} g \ \Big| \ \cF(s)} \leq \frac{1}{Y(s)} \expec \bra{ \frac{Y(T)}{\cE(-M^A) (T)}  \ \Big| \ \cF(s)} \leq \frac{1}{\cE(-M^A) (s)}, \quad \forall \, s \in [0, T], \quad \forall \, Y \in \cY,
\end{equation}
as follows from Proposition \ref{prop:mulp_decomp_local_mart_defl}, since $Y/ \cE(-M^A)$ is a nonnegative local martingale. In particular, $x \equiv x(g 1_{[T, \infty)}) \leq 1 < \infty$. Let $X \in \cX$ be a maximal in $[0, T]$ process such that $X(T) = g$. Furthermore, let $Z$ be the minimal hedge for $g 1_{[T, \infty)}$. Since $Z(0) = x \leq X(0)$ holds by definition of the hedging value and Theorem \ref{thm:odt}, and well as $Z(T) = g = X(T)$, maximality of $X$ implies that $X = Z$, i.e., $X$ is necessarily the maximal hedge of $g 1_{[T, \infty)}$. Using \eqref{eq:complete_help}, Theorem \eqref{thm:odt} implies that $X \leq 1 / \cE(- M^A)$. Set $N = \cE(- M^A) X$, and note that $N(T) = 1/2 + L(T)$, and that $N$ is a nonnegative bounded local martingale on $[0, T]$, i.e., an actual martingale. Additionally, in view of Lemma \ref{lem:numeraire_discount}, $N = x + M^H$ holds for $H \in \RKH(C)$. Since $L$ is also a martingale, it follows that $N (t) = 1/2 + L (t)$ holds for all $t \in [0, T]$. This implies $\bra{L, L} = \bra{M^H, L}  \equiv 0$, which  leads to $L \equiv 0$, reaching a contradiction. We conclude that the implication $(1) \Rightarrow (2)$ is valid.
\end{proof}

Consider a viable and complete market. During the proof of Theorem \ref{thm:ftap2}, it was established that the minimal hedge for European contingent claim involves no capital withdrawal. In fact, this is also the case for the minimal hedge associated with any $K \in \FV_\succeq$  such that
\[
x(K) = \bE  \left[ \int_0^\infty Y(t) \ud K(t) \right] < \infty,	
\]
where $Y$ is the unique local martingale deflator, on account of Theorem \ref{thm:ftap2}. Indeed, by \eqref{eq:minimal_hedge}, the process $X$ which minimally hedges $K$ satisfies
\begin{equation} \label{eq:complete_min_financing}
Y(s) X(s) = Y(s) K (s) + \bE \bra{ \int_{(s, \infty)} Y(t) \ud K(t) \ \Big| \ \cF(s ) }, \quad s \in \bR_+.
\end{equation}
Since the process $Y K - \int_0^\cdot Y(t) \ud K(t)  Y K = \int_0^\cdot K(t-) \ud Y(t) $ is a local martingale, it follows from \eqref{eq:complete_min_financing} that $Y X$ is also a local martingale. In view of Remark \ref{rem:odt_loc_mart}, the process $X$, which minimally hedges $K$, does so without any capital withdrawals.

\subsection{An example: Heath-Jarrow-Morton model}
\label{subsec:HJM}

We shall take now the index set $I$ to be  the nonnegative real line $\bR_+$, corresponding to all possible  \emph{maturities} for zero-coupon bonds,  instruments that pay off a single unit of currency at maturity. We shall illustrate within the context of such markets, the theory we have developed thus far. We start by  placing ourselves  in the Heath-Jarrow-Morton framework for the  prices of zero-coupon bonds.  To illustrate this background, let us denote by $\widetilde{P}(t; T)$ the  price at time $t \in \bR_+$ of a zero-coupon bond with maturity $T > t$. The idea is to model explicitly   the  evolution of \emph{forward rates}, which are formally obtained from bond prices via 
\[
f(t; T) = \frac{\partial \log \widetilde{P}}{\partial T} (t; T), \qquad 0 \leq t \leq T < \infty.
\]
In particular, $r(t) \equiv f(t; t)$ for $t \in \bR_+$ stands for the \emph{instantaneous short rate} 
over the infinitesimal interval $(t, t+ \ud t]$. Therefore, setting by convention $f(s; t) = f(t; t) = r(t)$ whenever $0 \leq t \leq s < \infty$, discounted zero-coupon bond prices should equal
\begin{align*}
P(t; T) &= \exp \pare{- \int_0^t r(u) \ud u } \widetilde{P} (t; T) \\
&= \exp \pare{- \int_0^t r(u) \ud u } \exp \pare{- \int_t^T f(t; u) \ud u } \\
&= \exp \pare{- \int_0^T f(t; u) \ud u}, \quad t \in \bR_+.
\end{align*}
We note that the above definition extends the ``life'' of bond prices $P(t; T)$ even when $t > T$, and that $P(t; T) = P(T; T)$ holds in this case. There is no cause for practical concern:  investing in the model's $T$-bond represented by prices $P(\cdot; T)$ after time $T$ will not result in any outcome (as   should be the case,  since   in reality the bond ceases to exist after its maturity).

Some formal definitions are necessary. Consider a collection  $W \equiv (W_\lambda; \, \lambda \in \Lambda)$    of independent Brownian motions, where  the  index set $\Lambda$ is at most countable. We recall  the stochastic rkHs setting of Remark \ref{rem:Brownian_rkHs}, which is tailor-made to fit  countable collections of independent Brownian motions as the ones used here. In accordance with this setting, we denote by $\ell^2 \equiv \ell^2_\Lambda$ the Hilbert space consisting of all sequences $y = (y_\lambda; \lambda \in \Lambda)$ with the property $\sum_{\lambda \in \Lambda} |y_\lambda|^2 < \infty$, and an inner product $\inner{\cdot}{\cdot}_{\ell^2}$ defined via 
\[
\inner{y}{z}_{\ell^2} = \sum_{\lambda \in \Lambda} y_\ell z_\ell, \qquad y = (y_\lambda; \lambda \in \Lambda) \in \ell^2, \quad z = (z_\lambda; \lambda \in \Lambda) \in \ell^2.
\]

We postulate now dynamics for the forward rates. With $\cB(\bR_+)$ denoting the Borel $\sigma$-algebra on $\bR_+$ and $\cP$ denoting the predictable $\sigma$-algebra on $\Omega \times \bR_+$, consider functions $f(0; \cdot) : \bR_+ \to \bR$, $\kappa: \Omega \times \bR_+ \times \bR_+ \to \bR$, as well as  $\sigma: \Omega \times \bR_+ \times \bR_+ \to \ell^2$ with the following properties:
\begin{itemize}
	\item $f(0; \cdot)$ is $\cB(\bR_+)$-measurable, and $\int_0^T |f(0; u)| \ud u < \infty$ holds  for all $T \in \bR_+$.
	\item The random fields $\kappa$ and $\sigma$ are $\cP \otimes \cB(\bR_+)$-measurable, and satisfy $\kappa(t; u) = 0$ and $\sigma(t; u) = 0$ whenever $0 \leq u < t$, as well as, $\bP$-a.e., 
	\[
	\int_0^T \int_0^T \pare{|\kappa(s; t)| + \norm{\sigma(s; t)}^2_{\ell^2} } \ud s \ud t < \infty, \quad \forall \, T \in \bR_+.
	\]
\end{itemize}
Under the above conditions, one may define a jointly measurable random field $f: \Omega \times \bR_+ \times \bR_+ \to \bR$, such that
\begin{equation}
\label{eq:for_rate_dynamics}
f (\cdot; T) = f (0; T) + \int_0^\cdot \kappa(t, T) \ud t + \int_0^\cdot \inner{\sigma(t, T)}{ \ud W(t)}_{\ell^2}
\end{equation}
holds for all $T \in \bR_+$. Note that the assumptions placed on $\kappa$ and $\sigma$ imply that $f(t; s) = f(s; s)$ holds whenever $0 \leq s \leq t < \infty$. 

We  introduce the 
\emph{discounted $T$-bond price processes}
\[
P(\cdot; T) \dfn \exp \pare{- \int_0^T f(\cdot; u) \ud u}, \quad T \in \bR_+,
\]
in accordance with the discussion at the start of the present Subsection. The thus-defined continuous semimartingales $P_T (\cdot) \equiv P (\cdot; T)$, indexed by the maturity parameter $T$ in the index set $I   \equiv \bR_+$, are the asset prices in the resulting bond market.

Given our assumptions, the stochastic Fubini theorem\footnote{See, for instance, \cite{MR2966093}; although a single Brownian motion (actually, continuous local martingale) is used in \cite{MR2966093}, the extension to countably many with the $\ell^2$-norm used here is  straightforward.} applies,  leading to the decomposition 
\[
\log P (\cdot; T) = - \int_0^T f(0; u) \ud u - \int_0^\cdot \kappa^*(t; T) \ud t - \int_0^\cdot \inner{\sigma^*(t; T)}{ \ud W(t)}_{\ell^2},
\]
where the processes $\kappa^*(\cdot; T)$ and $\sigma^*(\cdot; T)$ are defined via
\begin{equation}
\label{eq:int_for_rate_dynamics}
\kappa^*(\cdot; T) = \int_0^T \kappa (\cdot; u) \ud u, \quad \sigma^*(\cdot; T) = \int_0^T \sigma(\cdot; u) \ud u, \quad T \in \bR_+.
\end{equation}
We have therefore  $P(0; T) = \exp \pare{- \int_0^T f(0; u) \ud u}$ and, upon defining
\[
c (\cdot; S, T) \dfn \inner{\sigma^* (\cdot; S)}{\sigma^* (\cdot; T)}_{\ell^2}, \quad (S, T) \in \bR_+ \times \bR_+
\]
as well as
\begin{equation}
\label{eq:bond_drift_rate}
\alpha (\cdot; T) \dfn - \kappa^* (\cdot; T) + \frac{1}{2} \norm{\sigma^* (\cdot; T)}_{\ell^2}, \quad T \in \bR_+,
\end{equation}
it follows that
\[
P(\cdot; T) = P(0; T) \,  \cE \pare{\int_0^\cdot \alpha (t; T) \ud t - \int_0^\cdot \inner{\sigma^*(t; T)}{ \ud W(t)}_{\ell^2}}
\]
The setting of Remark \ref{rem:stoch_rkhs_continuous} applies here. Indeed, with $\clo (t) = \Leb(t) = t$, $t \in \bR_+$, the Lebesgue clock, the mapping $T \mapsto \sigma^* (\omega, t; T) \in \ell^2$ is continuous for $(\bP \times \Leb)$-a.e. $(\omega, t) \in \Omega \times \bR_+$. Then, the resulting bond market is viable, if and only if, the process of \eqref{eq:bond_drift_rate} satisfies, $\bP$-a.e.,
\begin{equation} \label{eq:viable_HJB}
\int_0^T \norm{\alpha (t; \bR_+)}^2_{c (t; \bR_+, \bR_+)} \ud t < \infty, \quad  \forall \,  T \in \bR_+.
\end{equation}
This  is exactly the structural condition \eqref{eq:struct_fin} in the present setting.

It is straightforward to check that all process $P(\cdot; T)$, $T \in \bR_+$ are local martingales if, and only if, for every $T \in \bR_+$, the dynamics in \eqref{eq:for_rate_dynamics}, \eqref{eq:int_for_rate_dynamics} satisfy the following condition:
\begin{equation} 
\label{eq:consist_HJM}
	\kappa(\cdot; T) = \inner{\sigma (\cdot; T)}{\sigma^* (\cdot; T)}_{\ell^2}, \qquad (\bP \otimes \Leb) \textrm{-a.e.}
\end{equation} 
The above relationships \eqref{eq:consist_HJM} between the processes $\kappa$ and $\sigma$, that describe the dynamics  of the forward rates in \eqref{eq:for_rate_dynamics}, constitute the so-called Heath-Jarrow-Morton \emph{drift restrictions}. These are derived in \cite{HJM:92} within the classical  framework by  assuming the existence of an equivalent local martingale measure  and expressing the dynamics of the model under this measure. Of course, the requirement \eqref{eq:viable_HJB} still results in a viable market, and is weaker than \eqref{eq:consist_HJM}, the latter being equivalent to asking that $\alpha \equiv 0$ in \eqref{eq:bond_drift_rate}.

\appendix

\section{Reproducing Kernel Hilbert Space}
\label{appsec:rkhs}

We record here certain elements of the theory of \textbf{reproducing kernel Hilbert space} (abbreviated as \textbf{rkHs} in the sequel). We take the route of defining rkHs starting from   given kernel, as opposed to obtaining the kernel from a given rkHS. By the Moore-Aronszajn theorem, these two viewpoints are equivalent; see \cite{MR51437}. There are plenty of sources that one may consult regarding the theory of rkHs; for example, \cite{BerThom:03} and \cite{MR3526117}.

The discussion below will take place in a deterministic setting. We consider an arbitrary nonempty index set $I$, and use $\Fin (I)$ (resp., $\Cou (I)$) to denote the collection of all non-empty subsets of $I$ with finite (resp., at most countably infinite) cardinality.  For the purposes of Appendix \ref{appsec:rkhs}, we shall take $c \equiv c_{II} \in \bR^{I \times I}$ to be a \textbf{kernel} on $I$, i.e.,
\begin{itemize}
	\item \emph{symmetric}: $c_{ij} = c_{ji}$ holds for $(i, j) \in I \times I$; and
	\item \emph{positive definite}: $\sum_{(i, j) \in J \times J} \theta_i c_{ij} \theta_j \geq 0$ holds for any $J \in \Fin(I)$ and $(\theta_i; \, i \in J) \in \bR^J$.
\end{itemize}
We shall use subscripts to denote the arguments of functions with domains that include $I$ or its subsets; for example, we shall write $c_{ij}$ instead of $c(i, j)$ for $(i, j) \in I \times I$. 

\subsection{Finite-dimensional rkHs} \label{subsec:rkh_fin}

We start by considering a nonempty index set $I$ of  finite cardinality. In this case, and with a slight abuse of notation, we also regard $c$ as a linear transformation on $\bR^I$ via the recipe $\bR^I \ni (\theta_j; j \in I) \equiv \theta \longmapsto c \theta \equiv \sum_{j \in I} \theta_jc_{I j} \in \rkh(c) \subseteq \bR^I$. While one may regard $c$ as a symmetric and positive-definite matrix, we shall make all the definitions that follow consistent with the infinite-dimensional setting developed later on. 

Let $\rkh (c) \subseteq \bR^I$ denote the linear span of the ``column'' functions $\set{c_{I j} \such \, j \in I}$, where we set
\begin{equation} \label{eq:column_funct}
c_{I j} \dfn (c_{ij} ; \, i \in I) \in \bR^I ,\quad j \in I.
\end{equation}
In effect, $\rkh(c)$ is the image of $c$. We  introduce the bilinear form $\inner{\cdot}{\cdot}_c: \rkh(c) \times \rkh(c) \to \bR$ via
\begin{equation} 
\label{eq:repr_ker_bil_form}
\inner{f}{h}_c \dfn \sum_{(i, j) \in I \times I} \theta_i c_{ij} \eta_j, \qquad \text{for} \quad f \equiv \sum_{j \in I} \theta_j c_{I j} = c \theta, \quad h \equiv \sum_{j \in I} \eta_j c_{I j}  = c \eta.
\end{equation} 
With the above notation, note the identities  $\sum_{i \in I} \theta_i h_i = \inner{f}{h}_c = \sum_{i \in I} \eta_i f_i$, implying that the quantity $\inner{f}{h}_c$ does not depend on the representation of $f$ or $h$ in $\rkh(c)$.

It is  straightforward to check that the bilinear form $\inner{\cdot}{\cdot}_c$ is an inner product on $\rkh(c)$, and has the so-called \emph{reproducing  kernel property} $\inner{c_{I i}}{f}_c = f_i$, for $f \in \rkh (c )$ and $i \in I$. The  finite-dimensional inner product space $(\rkh(c), \inner{\cdot}{\cdot}_c)$  defined in this manner, is the rkHs associated with $c$. We  introduce  the usual norm $\norm{f}_c \dfn \sqrt{\inner{f}{f}_c}$ for $f \in \rkh(c)$ on account of \eqref{eq:repr_ker_bil_form}; for future notational consistency, define $\norm{f}_c = \infty$ whenever $f \in \bR^I \setminus \rkh(c)$.

We denote by $\id_{\bR^I}$ the identity operator on $\bR^I$, and define
\begin{equation} \label{eq:rkhs_pre_image_fin_dim}
\theta^{f; n} \dfn \pare{c + (1/n) \id_{\bR^I}}^{-1} f, \qquad f \in \bR^I, \quad n \in \bN.
\end{equation}

\begin{lemma} 
[\textbf{Generalised inverse}]
\label{lem:norm_bare_hand}
With the above notation, we have
\begin{equation} \label{eq:rkhs_norm_fin_dim_alg}
\norm{f}^2_c =  \lim_{n \to \infty}   \uparrow   \big \langle \theta^{f; n}, f \big \rangle_{\bR^I}, \quad f \in \bR^I .
\end{equation}
In particular,  for $f \in \bR^I \setminus \rkh(c)$, both sides of \eqref{eq:rkhs_norm_fin_dim_alg}
are equal to infinity; on the other hand,  if $f \in \rkh(c)$,   we have 
\begin{equation} \label{eq:rkhs_pre_image_fin_dim_a}
f  = \sum_{j \in J} \theta^f_j c_{I j} = c \theta^f, \quad \text{where} \quad \theta^f \dfn \lim_{n \to \infty} \theta^{f; n} = \lim_{n \to \infty} (c + (1/n) \id_{\bR^I})^{-1} f .
\end{equation}
\end{lemma}

\begin{proof}
Let $u$ be a linear operator on $\bR^I$, unitary with respect to $\inner{\cdot}{\cdot}_{\bR^I}$ and such that $c = u^* d u$, where$u^*$ is the adjoint of $u$, and   $d$ is a positive diagonal operator. Define $J \dfn \set{j \in I \such d_{jj} > 0}$.
	
Let $f \in \bR^I$, and write $f = c \theta + \eta$ for some $\theta \in \bR^I$ and $\eta \in \bR^I$ such that $c \eta = 0$; then, $f \in \rkh(c)$ is equivalent to $\eta = 0$. We note that $c \eta = 0$ leads to $\pare{c + (1/n) \id_{\bR^I}}^{-1} \eta = n \eta$   for all $n \in \bN$. Let $\xi^n$  be the diagonal operator with $\xi^n_{ij} = \delta_{ij} / (\delta_{ij} + n^{-1})$ for $(i,j) \in J^2$, and $\xi^n_{ij} =0$ for $(i,j) \in I^2 \setminus J^2$. Straightforward algebra shows that $\pare{c + (1/n) \id_{\bR^I}}^{-1} c \theta = u^* \xi^n u \theta$ holds for $n \in \bN$. It follows from \eqref{eq:rkhs_pre_image_fin_dim} that $\theta^{f;n} = u^* \xi^n u \theta + n \eta$ also holds for $n \in \bN$.
	
Consider first  the case  $f \in \rkh(c)$, i.e., $\eta = 0$. Then, the sequence $(\xi^n; n \in \bN)$ converges to the diagonal operator $\xi$ with $\xi_{jj} = 1$ for $j \in J$, and $\xi_{jj} = 0$ for $j \in I \setminus J$; it follows that $\theta^f \dfn \lim_{n \to \infty} \theta^{f;n} = u^* \xi u \theta \in \bR^I$. Since $d \xi = d$, we deduce   $c \theta^f = u^* d u u^* \xi u \theta = u^* d \xi u \theta = u^* d u \theta = c \theta = f$. In particular,   $\lim_{n \to \infty} \inner{\theta^{f;n}}{f}_{\bR^I} = \inner{\theta^f}{f}_{\bR^I} = \norm{f}^2_c$ holds.
	
Suppose next, that $f \in \bR^I \setminus \rkh(c)$, i.e., $\eta \neq 0$. Then, $\lim_{n \to \infty} (1/n) \theta^{f;n} = \eta$ implies $\lim_{n \to \infty} (1/n) \big \langle \theta^{f; n}, f \big \rangle_{\bR^I} = \inner{\eta}{c \theta + \eta}_{\bR^I} = \norm{\eta}^2_{\bR^I} > 0$. We obtain $\lim_{n \to \infty} \inner{\theta^{f;n}}{f}_{\bR^I} = \infty = \norm{f}^2_c$, which completes the proof.
\end{proof}

The significance of Lemma \ref{lem:norm_bare_hand} is clear. The definition for $\theta^f$ in \eqref{eq:rkhs_pre_image_fin_dim_a} will always  ensure that the representation $f = c \theta^f$ holds whenever $f \in \rkh(c)$, even if $c$ (regarded as a linear transformation of $\bR^I$) fails to be invertible. The limiting procedure in \eqref{eq:rkhs_pre_image_fin_dim} should not be confounded with the Tychonoff regularization, used to obtain the Moore-Penrose pseudo-inverse of $f$ under $c$; this procedure would replace the so-defined $\theta^{f; n}$ by $\psi^{f; n} \dfn \pare{c^2 + (1/n) \id_{\bR^I}}^{-1} c f$, for $n \in \bN$. Then, using notation from the proof of Lemma \ref{lem:norm_bare_hand}, $\lim_{n \to \infty} \psi^{f; n} = u^* \xi u \theta$ holds, and implies that $\lim_{n \to \infty} \inner{\psi^{f; n}}{f}_{\bR^I}$ is always a finite real number; but this makes it impossible to recognise whether $f$ belongs to $\rkh (c)$, or not.

\subsection{General rkHs} \label{subsec:rkh_gen}

Now, assume that $I$ is an arbitrary nonempty index set. As  in \S \ref{subsec:rkh_fin}, we  consider the ``column'' functions $\set{c_{I j} \such \, j \in I}$ as in \eqref{eq:column_funct}. 

For any given $J \in \Fin(I)$, we denote by  $\rkh (c; J) \subseteq \bR^I$   the linear span of the columns $\set{c_{I j}; \, j \in J}$; and  again as in   \S \ref{subsec:rkh_fin},  we define on $\rkh (c; J)$ the bilinear form $\inner{\cdot}{\cdot}_{c; J}$ via $\inner{f}{g}_{c; J} \dfn \sum_{(i, j) \in J \times J} \theta_i c_{ij} \eta_j$, where $f = \sum_{j \in J} \theta_j c_{I j}$ and $h = \sum_{j \in J} \eta_j c_{I j}$. Thus $(\rkh (c; J), \inner{\cdot}{\cdot}_{c; J})$ becomes a finite-dimensional inner product space.

For arbitrary $J \in \Fin (I)$, $Q \in \Fin(I)$ with $J \subseteq Q$, the finite-dimensional Hilbert space $\big( \rkh (c; Q), \inner{\cdot}{\cdot}_{c; Q} \big)$ is an extension of $\big( \rkh (c; J), \inner{\cdot}{\cdot}_{c; J} \big)$. This means that $\rkh (c; J) \subseteq \rkh (c; Q)$ holds, and that $\inner{\cdot}{\cdot}_{c; J}$ is the restriction of $\inner{\cdot}{\cdot}_{c; Q}$ on the product space $\rkh (c; J) \times \rkh (c; J)$. We deduce that  an inner product $\inner{\cdot}{\cdot}_c$ can be   defined consistently on the vector space
\begin{equation} 
\label{eq:repr_ker_Fin}
\rkh (c; \Fin) \dfn \bigcup_{J \in \Fin (I)} \rkh (c; J) \subseteq \bR^I.
\end{equation} 
We  introduce also the associated norm $\rkh (c; \Fin) \ni f \mapsto \norm{f}_c \dfn \sqrt{\inner{f}{f}_{c}}$. By definition, we have once again 
the reproducing kernel property  
\begin{equation} 
\label{eq:repr_ker_propery}
\inner{c_{I i}}{f}_c = f_i, \qquad f \in \rkh (c; \Fin), \quad i \in I;
\end{equation} 
this implies  $|f_i| \leq \norm{c_{I i}}_c \norm{f}_c = \sqrt{c_{ii}} \norm{f}_c$, $ i \in I$, which establishes the continuity of the linear \textbf{evaluation functional} $\rkh (c; \Fin) \ni f \longmapsto f_i \in \bR$, for every $i \in I$.

The set $I$ does not necessarily have finite cardinality, so the resulting inner-product space $\pare{\rkh (c; \Fin), \inner{\cdot}{\cdot}_c}$ need not be complete. The following definition accounts for this fact.  We then define $\pare{\rkh (c), \inner{\cdot}{\cdot}_c}$, the \textbf{rkHs associated with the positive-definite kernel $c$}, as the Hilbert-space  completion of the inner-product space $\pare{\rkh (c; \Fin)  \inner{\cdot}{\cdot}_c}$   in \eqref{eq:repr_ker_Fin}, with the same notation for the extended inner product $\inner{\cdot}{\cdot}_c$ as before.

The above completed space $\rkh (c)$ is, in general, identified abstractly with equivalence classes of Cauchy sequences in $\rkh (c; \Fin)$. It turns out, however, that the rkHs $\rkh (c)$  has also another,  very concrete and useful,  description; this is  discussed in   \S \ref{subsec:rkh_alt} below. 

We   note that for any Cauchy sequence $(f^n; \, n \in \bN)$ in $\rkh (c; \Fin)$, the real-valued sequence $(f^n_i; \, n \in \bN)$ is Cauchy in $\bR$; this follows from the continuity of evaluation functionals, and implies  that the limit $f_i \dfn \lim_{n \to \infty} f^n_i$ exists for every $i \in I$. Therefore, the space $\rkh (c)$ can---and always will---be identified with a subset of $\bR^I$; indeed, the rkHs $\rkh (c)$ coincides with the subset of $\bR^I$ consisting of the  \emph{point-wise} limits of all Cauchy sequences in $\pare{\rkh (c; \Fin), \inner{\cdot}{\cdot}_c}$.

We further extend  the definition of $\rkh(c; J)$, from the case of $J \in \Fin(I)$ to that of an \emph{arbitrary subset} $J \subseteq I$, by setting it to be the $\norm{\cdot}_c$-closure in $\rkh(c)$ of the linear span of the column functions $\set{c_{Ij} \such j \in J}$. We also  observe the identity
\[
\rkh(c) \equiv \rkh(c; I) = \bigcup_{J \in \Cou(I)} \rkh(c; J).
\]
Indeed, the set-inclusion $\bigcup_{J \in \Cou(I)} \rkh(c; J) \subseteq \rkh(c)$ is obviously true. Concerning the reverse inclusion, we  note that   given any $f \in \rkh(c)$,    any $\rkh(c; \Fin)$-valued sequence $(f^n; \, n \in \bN)$ converging to $f$, and any  sequence $(J^n; \, n \in \bN)$    in $\Fin(I)$ with the property   $f^n \in \rkh(c; J_n)$ for every  $n \in \bN$, we have  clearly 
$f \in \rkh(c; Q)$, where $Q \equiv \bigcup_{n \in \bN} J^n \in \Cou(I)$.

\subsection{Restrictions and projections} \label{subsec:rkh_rest_proj}

For arbitrary $J \subseteq I$, denote by  $f_J \equiv \pare{f_i ; i \in J} \in \bR^J$ the restriction of $f \in \bR^I$ to $J$; and by $c_{JJ} \equiv (c_{ij}; \, (i, j) \in J \times J)$   the restriction of $c$ to $J \times J$.

\begin{lemma} \label{lem:rkhs_proj_rest}
For arbitrary $J \subseteq I$, the mapping $\rkh (c; J) \ni f \mapsto f_J \in \rkh (c_{JJ})$ is well-defined, and a Hilbert space isomorphism.
\end{lemma}

\begin{proof}
First, we assume that $J \in \Fin(I)$; then  $f = \sum_{i \in J} \theta_j c_{I j} \in \rkh (c; J)$ holds for $(\theta_j; \, j \in J) \in \bR^J$ if and only if $f_J = \sum_{i \in J} \theta_j c_{J j} \in \rkh (c_{JJ})$; and by definition, we have also then 
	\[
	\norm{f}^2_{c; J} = \sum_{(i,j) \in J \times J} \theta_i c_{ij} \theta_j = \norm{f_J}^2_{c_{JJ}}.
	\]
	The case of an arbitrary subset $J$ follows by a straightforward density argument, upon 
	recalling  the continuity  of linear evaluation functionals.   	
\end{proof}

It follows from Lemma \ref{lem:rkhs_proj_rest} that $\rkh (c_{JJ}) \subseteq \bR^J$ consists exactly of restrictions  of the elements of $\rkh (c; J) \subseteq \bR^I$ on the subset $J$; and that the coordinates $\pare{f_i; \, i \in I \setminus J}$ of any $f \in \rkh (c; J)$,  are determined entirely by $f_J \equiv \pare{f_i; \, i \in J}$ and by the structure of the kernel $c$.

\smallskip

For an arbitrary  subset  $J \subseteq I$, we denote by $\pi_{c; J} (f) \in \bR^I$ the $\inner{\cdot}{\cdot}_c$-projection of $f \in \rkh(c)$ on $\rkh(c; J)$. Since the reproducing kernel property $\inner{c_{I j}}{f}_c = f_j$ of \eqref{eq:repr_ker_propery} holds for all $j \in J$, and the linear span of $\set{c_{Ij}; \, j \in J}$ is dense in $\rkh(c; J)$, we have $\pi_{c; J} (f)_j = f_j$, for all $f \in \rkh(c)$ and $j \in J$. In fact, $\pi_{c; J}(f)$ is the unique element $h \in \rkh(c; J) \subseteq \bR^I$, whose restriction $h_J$ on $J$ coincides with the restriction $f_J$ of $f$ on $J$.

As a consequence of the above discussion and of Lemma \ref{lem:rkhs_proj_rest},  we note that  $f_J \in \rkh(c_{JJ})$ and $\norm{f_J}_{c_{JJ}} = \norm{\pi_{c; J} (f)}_{c} \leq \norm{f}_c$ hold for $f \in \rkh(c)$, $J \subseteq I$. Using the index set $Q \subseteq I$ in place of $I$, we obtain the inequality $\norm{f_J}_{c_{JJ}} \leq \norm{f_Q}_{c_{QQ}}$  whenever $J \subseteq Q \subseteq I$, $f_Q \in \rkh(c_{QQ})$. In fact, it is straightforward to check 
\begin{equation} 
\label{eq:rkh_norm_incr}
f \in \bR^I, \quad J \subseteq Q \subseteq I \quad \Longrightarrow \quad  \norm{f_J}_{c_{JJ}} \leq \norm{f_Q}_{c_{QQ}} \leq \norm{f}_{c} .
\end{equation}
We  use here  the convention  that the norms,  of those elements  which do not belong to the corresponding spaces,   are understood to be equal to infinity.

\subsection{An alternative description of rkHs}
\label{subsec:rkh_alt}

The following result will be used as the basis for an alternative,  concrete characterisation of  the rkHs $\rkh(c)$, and for its Hilbert-space  structure. This characterization is developed in Remark \ref{rem:rkhs_norm_charact} below.

\begin{lemma} \label{lem:rkhs_norm_charact}
It holds that
\begin{equation} \label{eq:rkhs_norm_charact}
\sup_{J \in \Fin(I)} \norm{f_J}_{c_{JJ}} = \max_{Q \in \Cou(I)} \norm{f_Q}_{c_{QQ}} = \norm{f}_c, \quad f \in \bR^I.
\end{equation}
\end{lemma}

\begin{proof}
On the strength of  \eqref{eq:rkh_norm_incr}, we have 
the string of inequalities 
\begin{equation} \label{eq:rkhs_norm_charact_too}
\nu(f) \dfn \sup_{J \in \Fin(I)} \norm{f_J}_{c_{JJ}} \leq \sup_{Q \in \Cou(I)} \norm{f_Q}_{c_{QQ}} \leq \norm{f}_c, \quad f \in \bR^I.
\end{equation}
Let $(J^n; \, n \in \bN)$ be a sequence in $\Fin(I)$ such that $\lim_{n \to \infty} \norm{f_{J^n}}_{c_{J^n J^n}} = \nu(f)$. In view of \eqref{eq:rkh_norm_incr}, we can choose $(J^n)_{n \in \bN}$ to be nondecreasing. With $R \dfn  \bigcup_{n \in \bN} J^n \in \Cou(I)$, the inequalities of   \eqref{eq:rkh_norm_incr} imply again  that $\nu(f) \leq \norm{f_R}_{c_{R R}} \leq \norm{f}_c$.
	
If $\nu(f) = \infty$, then \eqref{eq:rkhs_norm_charact} follows  directly from the string of inequalities \eqref{eq:rkhs_norm_charact_too}. Thus, for the remainder of the proof, assume that $\nu(f) < \infty$. For each $n \in \bN$, Lemma \ref{lem:rkhs_proj_rest} implies the existence of $g^n \in \rkh (c; J^n) \subseteq \rkh (c)$ with $g^n_{J^n} = f_{J^n}$ and $\norm{g^n}_c = \norm{f_{J^n}}_{c_{J^n J^n}}$. Furthermore, for $m \leq n$, since $g^m_{J^m} = f_{J^m} = g^n_{J^m}$, the discussion in  \S \ref{subsec:rkh_rest_proj} implies that $g^m$ is the $\inner{\cdot}{\cdot}_c$-projection of $g^n$ on $\rkh (c; J^m)$; therefore, 
$\norm{g^n - g^m}^2_c = \norm{g^n}^2_c - \norm{g^m}^2_c = \norm{f_{J^n}}^2_{c_{J^n J^n}} - \norm{f_{J^m}}^2_{c_{J^m J^m}}$. Given that $\lim_{n \to \infty} \norm{f_{J^n}}_{c_{J^n J^n}} = \nu(f) < \infty$, it follows that $(g^n; n \in \bN)$ is a Cauchy sequence  in $(\rkh(c), \inner{\cdot}{\cdot}_c)$; we denote its limit by  $g \in \rkh(c)$, and notice that $\norm{g}_c = \nu(f)$. We claim that $g = f$; once this has been established, \eqref{eq:rkhs_norm_charact} will follow  from the inequalities $\nu(f) \leq \norm{f_R}_{c_{R R}} \leq \norm{f}_c$ already discussed.

We proceed to show    $g = f$. We fix an arbitrary index $i \in I$, and follow the argument of the previous paragraph with the sets $J^n \cup \set{i}$ in place of $J^n$, obtaining along the way a new Cauchy sequence $(h^n; n \in \bN)$ in place of $(g^n; n \in \bN)$, and a new limit $h \in \rkh(c)$ in place of $g \in \rkh(c)$. Observe that we still have $\norm{h}_c = \nu(f)$. Since $i \in J^n \cup \set{i}$, we note that $h^n_i = f_i$ holds for all $n \in \bN$; this    gives $h_i = f_i$, because the evaluation functionals in $\rkh(c)$ are continuous. Furthermore, for every $m \leq n$,   $g^m$ is the $\inner{\cdot}{\cdot}_c$-projection of $h^n$ on $\rkh (c; J^m)$. In turn, this implies that $g^m$ is the $\inner{\cdot}{\cdot}_c$-projection of $h$ on $\rkh (c; J^m)$, for each $m \in \bN$; 
thus $g$ is the $\inner{\cdot}{\cdot}_c$-projection of $h$ on $\rkh (c; R)$. But since $\norm{g} = \norm{h}$, we have $g = h$, which implies  $g_i = h_i = f_i$. Since $i \in I$ is arbitrary  we obtain $g = f$, concluding the proof.
\end{proof}

\begin{remark} \label{rem:rkhs_norm_charact}
An immediate consequence of Lemma \ref{lem:rkhs_norm_charact}, is the equivalence of  the following statements for an arbitrary element $f \in \bR^I$:
\begin{enumerate}
	\item $f \in \rkh(c)$.
	\item $f_Q \in \rkh(c_{QQ})$ for every $Q \in \Cou(I)$.
	\item $f_J \in \rkh(c_{JJ})$ for every $J \in \Fin(I)$, and $\sup_{J \in \Fin(I)} \norm{f_J}_{c_{JJ}} < \infty$.
\end{enumerate}

\smallskip
Another important aspect of Lemma \ref{lem:rkhs_norm_charact}, is that it provides an alternative characterization of  the space $\rkh(c) \subseteq \bR^I$ and of its inner-product structure. To present this extension, we define
\[
\nu_c(f; J) \dfn    \lim_{n \to \infty} \uparrow \sqrt{ \inner{f_J}{(c_{JJ} + (1/n) \id_{\bR^J})^{-1} f_J}_{\bR^J} }, \quad f \in \bR^I, \  \  J \in \Fin(I) 
\]
as in     \S \ref{subsec:rkh_fin} and by analogy with \eqref{eq:rkhs_norm_fin_dim_alg}, then   set 
\begin{equation}
\label{eq:rkHs_norm}
\nu_c (f) \dfn \sup_{J \in \Fin(I)} \nu_c(f; J), \quad f \in \bR^I.
\end{equation}
A combination of \eqref{eq:rkhs_norm_fin_dim_alg} and Lemma \ref{lem:rkhs_norm_charact} leads to the identifications
\[
\rkh(c) = \set{f \in \bR^I \such \nu_c(f) < \infty} , \qquad  
\nu_c (\cdot) = \norm{\cdot}_c
\]
for the rkHs $\rkh(c)$ and for the norm of \eqref{eq:rkHs_norm}, respectively.  The inner-product $\inner{\cdot}{\cdot}_c$ can then be recovered via polarization, namely,
\[
\inner{f}{g}_c \dfn \frac{1}{4} \big( (  \nu_c(f+g)  )^2 +  ( \nu_c(f-g) )^2 \big), \quad (f,g) \in  \bR^I \times \bR^I .
\]

The construction   described right above, constitutes a very direct  procedure, algebraic and limiting in nature, for obtaining the rkHs $\rkh(c)$; it does not involve any abstract completion. This approach is used in Section \ref{sec:stoch_rkhs} for defining stochastic counterparts of these notions.
\end{remark}

\subsection{Continuity of elements in rkHs}

When the 
index set $I$ carries a topology, it is of interest  to study the continuity properties of the elements of $\rkh(c)$ since these are,  in particular, elements of the function space $\bR^I$. Clearly, a necessary condition for all elements of $\rkh(c)$ to be continuous, is the continuity of the ``column'' functions $c_{Ij} \equiv (c_{ij}; i \in I) \in \rkh(c)$, for every $j \in I$. In fact, the next result shows that not much more is needed.

\begin{lemma} \label{lem:rkhs_continuous}
Suppose that $I$ is endowed with a topology, and assume that:
\begin{itemize}
	\item $c_{Ij} \equiv (c_{ij}; i \in I) \in \rkh(c)$ is continuous, for every $j \in I$;
	\item the mapping $I \ni j \mapsto c_{jj} \in \bR$ is locally bounded: for every $i \in I$, there exists an open neighbourhood $J(i) \subseteq I$ with $\sup_{j \in J (i)} c_{jj} < \infty$.
\end{itemize}
Then, all elements in $\rkh (c)$ are continuous.
\end{lemma}

\begin{proof}
The fact that the function $c_{Ij}$ is continuous for every $j \in I$, implies that all elements of $\rkh(c; \Fin)$ are continuous.

Fix $f \in \rkh (c)$, $i \in I$ and a net $(i^\lambda; \ \lambda \in \Lambda)$ in $I$, where $\Lambda$ is a directed set, converging to $i$. Fix an open neighbourhood $J(i) \subseteq I$ such that $\ell(i) \dfn \sup_{j \in J (i)} \sqrt{c_{jj}} < \infty$. Consider a sequence $(f^n; \, n \in \bN)$ in $\rkh (c; \Fin)$ such that $\lim_{n \to \infty} \norm{f^n - f}_{c} = 0$. For $k \in \bN$, let $n_k \in \bN$ be large enough so that $\norm{f^{n_k} - f}_{c} \leq (4 k \ell(i))^{-1}$ holds. Then, pick $\mu_k \in \Lambda$ with the property that $i^\lambda \in J(i)$ and $|f^{n_k}_{i^\lambda} - f^{n_k}_i| \leq (2k)^{-1}$ holds whenever $\lambda \geq \mu_k$, and observe
\begin{align*}
|f_{i^\lambda} - f_i| &\leq |f_{i^\lambda} - f^{n_k}_{i^\lambda}| + |f^{n_k}_{i^\lambda} - f^{n_k}_i| + |f^{n_k}_i - f_i| \\
&\leq \pare{\sqrt{c_{i^\lambda i^\lambda}} + \sqrt{c_{i i}}} \norm{f^{n_k} - f}_{c} + |f^{n_k}_{i^\lambda} - f^{n_k}_i| \leq 1 / k
\end{align*}
 for all $\lambda \geq \mu_k$.    It follows that  $(f_{i^\lambda}; \, \lambda \in \Lambda)$ converges to $f_i$, completing the argument.
\end{proof}

\begin{remark} \label{rem:rkhs_continuous_dense}
In addition to the assumptions of Lemma \ref{lem:rkhs_continuous}, suppose that there exists a countable dense subset $Q$ of $I$. Using notation from \S   \ref{subsec:rkh_gen} and \S  \ref{subsec:rkh_rest_proj}, whenever $f \in \rkh(c)$ and $g \in \rkh(c; Q)$ are such that $\pi_{c;Q}(f) = g$, we have in this case $f = g$; this is  because both $f$ and $g$ are continuous, $f_Q = g_Q$, and $Q$ is dense in $I$. It follows then that $\rkh(c) = \rkh(c; Q)$, i.e., that $\rkh(c)$ is Hilbert-isomorphic to $\rkh(c_{QQ})$, and $\norm{f}_{c} = \norm{f_Q}_{c_{QQ}}$ holds for all $f \in \rkh(c)$.
\end{remark}

\bibliographystyle{amsalpha}
\bibliography{biblio}
\end{document}

%% file: defs.tex
\newcommand{\pare}[1]{\left(#1\right)}	
\newcommand{\bra}[1]{\left[#1\right]}	
\newcommand{\set}[1]{\left\{#1\right\}}	
\newcommand{\abs}[1]{\left|#1\right|}	
\newcommand{\inner}[2]{\left \langle #1, #2 \right \rangle}	
\newcommand{\norm}[1]{\left \Vert #1 \right \Vert}	

\newcommand*{\dfn}{\mathrel{\rlap{%
			\raisebox{0.3ex}{$\cdot$}}%
		\raisebox{-0.3ex}{$\cdot$}}%
	=}

\DeclareMathOperator*{\esssup}{ess\,sup}

\newcommand{\cB}{\mathcal{B}}
\newcommand{\cE}{\mathcal{E}}
\newcommand{\cF}{\mathcal{F}}

\newcommand{\cK}{\mathcal{K}}

\newcommand{\cP}{\mathcal{P}}

\newcommand{\cX}{\mathcal{X}}
\newcommand{\cY}{\mathcal{Y}}

\newcommand{\bE}{\mathbb E}

\newcommand{\bL}{\mathbb L}

\newcommand{\bN}{\mathbb N}
\newcommand{\bP}{\mathbb P}

\newcommand{\bR}{\mathbb R}

\newcommand{\bX}{\mathbb X}
\newcommand{\bY}{\mathbb Y}
\newcommand{\bZ}{\mathbb Z}

\newcommand{\Leb}{\mathsf{Leb}}

\newcommand{\lzp}{\bL^0_+}

\newcommand{\one}{1}

\newcommand{\such}{\, | \,}

\newcommand{\ud}{\mathrm d}

\newcommand{\jii}{j \in \set{1, \ldots, n}}

\newcommand{\clo}{O}
\newcommand{\pclo}{\bP \otimes \clo}
\newcommand{\ppclo}{\pare{\pclo}}

\newcommand{\id}{\mathsf{id}}

\newcommand{\expecp}{\bE^\bP}

\newcommand{\co}{\mathsf{c}}

\newcommand{\dcX}{{}^{d} \kern-0.23em \cX}
\newcommand{\delcX}{{}^{\delta} \kern-0.23em \cX}
\newcommand{\epscX}{{}^{\epsilon} \kern-0.23em \cX}
\newcommand{\onecX}{{}^{1} \kern-0.23em \cX}
\newcommand{\dX}{{}^{d} \kern-0.23em X}
\newcommand{\delX}{{}^{\delta} \kern-0.23em X}
\newcommand{\epsX}{{}^{\epsilon} \kern-0.23em X}
\newcommand{\delZ}{{}^{\delta} \kern-0.23em Z}
\newcommand{\delpi}{{}^{\delta} \kern-0.17em \pi}
\newcommand{\delnu}{{}^{\delta} \kern-0.17em \nu}

\newcommand{\Fin}{\mathsf{Fin}}
\newcommand{\Cou}{\mathsf{Cou}}
\newcommand{\rkh}{\mathcal{R}}
\newcommand{\RKH}{\mathsf{R}}

\newcommand{\FV}{\mathsf{FV}}
\newcommand{\cFV}{\mathsf{cFV}}

\newcommand{\dnorm}[1]{\left \lceil \mkern-5mu \left \lceil #1 \right \rceil \mkern-5mu \right \rceil}

\newcommand{\cSem}{\mathsf{cS}}

\newcommand{\cXs}{\cX_{\mathsf{s}}}
\newcommand{\xs}{x_{\mathsf{s}}}